\documentclass[10pt]{amsart}
\usepackage{graphicx} 
\usepackage{hyperref}
\title{Dual of algebraic geometry codes from Hirzebruch Surfaces}
\author{Alix Barraud}
\address{CNRS; Institut de Mathématiques de Bordeaux - UMR 5251, Université de Bordeaux
351 Cours de la Libération
33405 TALENCE CEDEX, France}
\email{alix.barraud@math.u-bordeaux.fr}
\date{}

\usepackage{mathtools}
\usepackage{tikz-cd}
\usepackage{amsthm}
\usepackage{amsfonts}
\usepackage{amssymb}
\usepackage{mathrsfs}
\usepackage{amsmath}
\usepackage{a4wide}
\usepackage[all]{xy}
\usepackage[Algorithme]{algorithm}
\usepackage[noend]{algpseudocode}
\usepackage{ stmaryrd }
\usepackage{soul}
\usepackage{enumitem}
\usepackage[T1]{fontenc}
\usepackage{dsfont}

\usepackage{appendix}

\usepackage{tikz}
\usetikzlibrary{calc,shapes.arrows,decorations.pathreplacing}
\usetikzlibrary{fadings}
\usetikzlibrary{math}

\usepackage{xcolor}

\newcommand\Lis{L_*}

\DeclareMathOperator{\Div}{Div}

\DeclareMathOperator{\Pic}{Pic}

\DeclareMathOperator{\ev}{ev}
\DeclareMathOperator{\F}{\mathbb{F}}
\DeclareMathOperator{\wt}{wt}
\DeclareMathOperator{\Image}{Im}
\DeclareMathOperator{\Supp}{Supp}
\DeclareMathOperator{\Lrr}{\mathcal{L}}

\DeclareMathOperator{\RS}{RS} 
\DeclareMathOperator{\PRS}{PRS}

\DeclareMathOperator{\WARM}{WARM}
\DeclareMathOperator{\RM}{RM}
\DeclareMathOperator{\CSS}{CSS}
\newcommand\ie{\emph{i.e.,~}}

\newtheorem{theorem}{\textsf{Theorem}}[section]
\newtheorem{lemma}[theorem]{\textsf{Lemma}}
\newtheorem{corollary}[theorem]{\textsf{Corollary}}
\newtheorem{proposition}[theorem]{\textsf{Proposition}}

\theoremstyle{definition}
\newtheorem{definition}[theorem]{\textsf{Definition}}
\newtheorem{remark}[theorem]{\textsf{Remark}}
\newtheorem{example}[theorem]{\textsf{Example}}
\newtheorem{notation}[theorem]{\textsf{Notation}}
\newtheorem{convention}[theorem]{\textsf{Convention}}


\begin{document}




\begin{abstract}
In this paper, we give an explicit form for the dual of the algebraic geometry code $C_e(a,b)$ defined on an Hirzebruch surface $\mathcal{H}_e$ and parametrized by the divisor $aS_e + bF_e$, where $a,b\in\mathbb{N}$ and $S_e$ and $F_e$ generate the Picard group $\Pic( \mathcal{H}_e)$. Notably, we compute the minimum distance of $C_e(a,b)^\perp$. One of the main ingredient for our study is a new explicit form of the code $C_e(a,b)$ which we provide at the beginning of the paper. We also investigate some puncturing of $C_e(a,b)$, recovering other previously studied AG codes from toric surfaces. Finally, we provide a sufficient condition for orthogonal inclusions between the codes $C_e(a,b)$, and construct CSS quantum codes from them.
\end{abstract}

\sloppy
\maketitle

\section*{Introduction}
After the first construction of algebraic geometry codes (AG codes for short) on algebraic curves by Goppa \cite{goppa1981codes}, a generalisation on higher dimensional varieties was proposed by Manin and Vl\u adu\c t \cite{vladut1984linear}, and S.~H.~Hansen \cite{hansen2001error}. However, compared to codes from curves, very few is known on codes from higher dimensional varieties. On algebraic surfaces, for example, computing the dimension requires a more complex form of the Riemann--Roch theorem \cite[V. Theorem 1.6]{Hart13}, while the difference of dimension between points and divisors of a surface makes it difficult to give a general lower bound for the minimum distance. One can cite generic lower bounds given by Aubry \cite{aubry1992algebraic, aubry2006reed}, using the maximum number of rational points on a curve, and, more recently, by Couvreur, Lebacque and Perret \cite{Couv20}, with arguments involving intersection theory applied to specific linear systems of divisors (so-called "$\mathcal{P}$-interpolating"). Because of the lack of general results, the usual approach to find new algebraic geometry codes on surfaces and their parameters is to focus on a specific family of surfaces. Estimating the parameters of such codes boils down to the study of the arithmetic and the properties of these surfaces. To name a few, codes on quadric surfaces \cite{edoukou2008codes, couvreur2013evaluation} were considered, as well as on abelian surfaces \cite{aubry2021algebraic}, Del Pezzo surfaces \cite{blache2024construction}, fibered surfaces \cite{aubry2021bounds}, or even on specific blow-ups of the projective plane \cite{ballico2013codes} \cite{couvreur2011construction}. Toric surfaces were also investigated, originally by J.~P.~Hansen in \cite{hansen2002toric} and later by various authors in \cite{little2006toric, ruano2009structure, hansen2012quantum} for instance. More recently, Nardi \cite{nardi2019algebraic} combined these toric codes with a projective point of view by considering minimal Hirzebruch surfaces and evaluating at every rational points, including those "at infinity". Nardi then computes parameters of these codes with the help of combinatorics arguments and Gröbner bases, for parameterizations that go beyond the injective cases allowing a wider range of codes to be considered, including some defined on a small base field. 

Current knowledge on the duals of algebraic geometry codes also gets thin on higher dimensional varieties. On curves, differential codes are introduced \cite[Chapter 2]{Stic09} by evaluating local components of differentials forms at each rational point, namely their residues. It can be shown, using the residue formula, that these constitute the duals of corresponding "functional" AG codes. Then, the duality theorem on curves \cite[1.5.14]{Stic09}, allows one to express any differential code as a functional code. Thus, the set of all AG codes on a given curve coincides with the set of their duals, and it is relatively easy to explicitly describe the dual of a given AG code on a curve. As a consequence, generic lower bounds for the minimum distance of algebraic geometry codes on curves also apply for their duals, and studying the minimum distance also gives informations on the dual distance in most cases. This is no longer true on higher dimensional algebraic varieties, because of the lack of generic non-trivial lower bounds for the minimum distance and the fact that the dual of an AG code is no longer an AG codes in general \cite[Proposition 6.10]{grant2024differential}. For algebraic surfaces, Couvreur \cite{Couv09} constructed residues of differential $2$-forms along a curve ($1$-codimensional residue) and on a point ($2$-codimensional), allowing him to construct differential AG codes on surfaces. He proved that these codes can always be expressed as functional AG codes, and vice-versa. He also showed that duals of AG codes from surfaces are no longer AG codes in general, but can be expressed as a finite sum of AG codes \cite{Couv11}. However, the proof is not constructive and we have no clue about the minimum number of AG codes that are necessary to express a dual. As for the dual distance, a generic lower bound was given \cite[Theorem 6.1]{Couv11} but it is not always non-trivial and requires a lot of knowledge about the considered surface. As a consequence, a few duals of AG codes on surfaces are known, and their minimum distances remain difficult to estimate.

Finally, we point out that knowing the duality theory of AG codes can turn out to be useful in different contexts. For instance, duals of AG codes on curves can be used for decoding \cite[8.5]{Stic09}, and this decoding algorithm could be adapted to AG codes defined on algebraic surfaces. Furthermore, advances in the topic on quantum codes gave a new interest  for the study of the dual of various codes. Indeed, the CSS construction developed by Calderbank and Shor \cite{calderbank1996good}, and Steane \cite{steane2002enlargement} gives a quantum code from a pair of linear codes. Parameters of these quantum codes can be estimated by knowing those of the two linear codes that were used, including their dual distances. In this context, algebraic geometry codes on curves were considered, leading to good quantum codes \cite{la2017good}.

\subsection*{Contributions and organisation of the paper}
This article can be seen as a follow-up to Nardi's work \cite{nardi2019algebraic}, and hopes to pave the way toward studies of duals of AG codes from more families of surfaces. 

In Section \ref{section:1}, we recall some tools from the algebraic geometry of surfaces, mainly concerning their divisors. We then provide a geometric and "visual" (Figure \ref{fig:HirzebruchDiviseurs}) interpretation of many features of the Hirzebruch surfaces $\mathcal{H}_e$, before giving an explicit form of their Riemann--Roch spaces in Propositions \ref{prop:RRHirz} and \ref{prop:RRgeneral}.

In Section \ref{section:2}, we briefly recall some prerequisites on linear codes and AG codes. We then give a new explicit form for an AG code $C_e(a,b)$ on Hirzebruch surfaces (Proposition \ref{prop: Code Hirzebruch}), using tensor product and Reed--Solomon codes. After recalling the parameters of $C_e(a,b)$ in Proposition \ref{prop:paramcode}, its explicit form is then used to compute its dual distance in Theorem \ref{prop:dualdistance} with the help of \emph{check-product} of codes (see Definition \ref{def:checkprod}).

By puncturing $C_e(a,b)$ at the coordinates of the point at infinity, we obtain a code $C_{\mathbb{A},e}(a,b)$ of length $q^2$ whose parameters and dual are studied in Section \ref{section:3}, in particular in Proposition \ref{prop:parametreq} and Theorems \ref{prop:dualdistanceq} and \ref{theo:dualq}. The dual $C_{\mathbb{A},e}(a,b)^\perp$ being isomorphic to a big subspace of $C_e(a,b)^\perp$, it allows us in Section \ref{section:4} to finally compute an explicit form of $C_e(a,b)^\perp$ (see Theorem \ref{theo:dual}).

In Section \ref{section:equivalence}, we prove some technical results to obtain pairs of AG codes defined on $\mathcal{H}_e$, one containing the other. We also use the dual of $C_e(a,b)^\perp$ to construct pairs of orthogonal AG codes defined on an Hirzebruch surface. Finally, we conclude the paper by using these pairs of linear codes to construct CSS quantum codes in Section \ref{section:css}.

\subsection*{Note} In all the paper, we fix $\F_q$ a finite field of any positive characteristic and of cardinality $q$. For two integers $k,l \in \mathbb{Z}$, we will denote $\llbracket k,l \rrbracket \coloneqq \{ x \in \mathbb{Z} \, | \, k \leq x \leq l \}$.
\section{Basic properties of Hirzebruch surfaces}\label{section:1}

We describe in this section the Hirzebruch surfaces, denoted $\mathcal{H}_e$ for $e \in \mathbb{N}$. Our main goal is to provide a visual interpretation of these varieties (see Figure \ref{fig:HirzebruchDiviseurs}), as well as an explicit description of their Riemann--Roch spaces (see Subsection \ref{ss:RR}).

\subsection{Algebraic surfaces and their geometry} Let $\mathcal{S}$ be a smooth projective and geometrically integral surface over $\mathbb{F}_q$, and denote $\F_q(\mathcal{S})$ its function field. In what follows, we will introduce basic notions on the geometry of algebraic surfaces, the known results being presented here without proof. For more details, we refer the reader to \cite{Liu02, Shaf13, Hart13}.
\subsubsection{Divisors on a surface}
 A \emph{prime divisor} $C$ of $\mathcal{S}$ is an irreducible and reduced curve of $\mathcal{S}$, and we call \emph{divisor} a formal sum $D \coloneqq \sum_{i=0}^s n_iC_i$ where $s \in \mathbb{N}$, and for all $i \in \llbracket 0,s \rrbracket$, $n_i \in \mathbb{Z}$ and $C_i$ is a prime divisor. We denote $\Div(\mathcal{S})$ the abelian group of all divisors of $\mathcal{S}$. We say that $D=\sum_{i=0}^s n_i C_i \in \Div(\mathcal{S})$ is \emph{positive} when $n_i \geq 0$ for all $i$; and for all $D_1, D_2 \in \Div(\mathcal{S})$, we write $D_1 \geq D_2$ when their difference $D_1 - D_2$ is positive. Finally, for $D = \sum_{i=0}^s n_i C_i \in \Div(\mathcal{S})$ where the $C_i$'s are distinct prime divisors, we call $\Supp(D) \coloneqq \sum_{i=0}^s C_i$ the \emph{support} of $D$, and we say that a point $P$ of $\mathcal{S}$ is \emph{in the support of} $D$ if there exists $i \in \llbracket 1, s \rrbracket$ such that $P \in C_i$. We will call \emph{curve} on $\mathcal{S}$ a positive divisor, \ie a divisor $C \in \Div( \mathcal{S})$ such that $D \geq 0$.

\subsubsection{Linear equivalence and Riemann--Roch space}

Let $f \in \F_q(\mathcal{S}) \setminus \{ 0 \}$. We associate to $f$ its divisor $(f) = (f)_0 - (f)_\infty \in \Div(\mathcal{S})$ where $(f)_0 \geq 0$ and $(f)_\infty\geq 0$ are the divisors associated to the zero and the pole locus, respectively, both accounting for multiplicity so that for all $f_1$, $f_2 \in \F_q(\mathcal{S}) \setminus \{ 0 \}$ we have $(f_1 f_2) = (f_1) + (f_2)$. Such a divisor is called \emph{principal}, and two divisors $D_1$ and $D_2 \in \Div(\mathcal{S})$ are \emph{linearly equivalent} if their difference $D_1 - D_2$ is principal, denoted $D_1 \simeq D_2$. Principal divisors form a subgroup of $\Div(\mathcal{S})$, and the \emph{Picard group} $\Pic(\mathcal{S})$ of $\mathcal{S}$ is defined as the quotient
\[
\Pic(\mathcal{S}) \coloneqq \Div( \mathcal{S}) / \simeq \, .
\] 
The \emph{Riemann--Roch space} $L(D)$ associated to a divisor $D \in \Div(\mathcal{S})$ is the $\F_q$-vector space
\[
L(D) \coloneqq \{ f \in \F_q(\mathcal{S})\setminus\{0\} | (f) \geq -D \} \cup \{ 0 \}.
\]
For all divisor $D_2 \simeq D_1$ with $D_2+ (f) = D_1$, we have $fL(D_2) \coloneqq \{ fg_2 \mid g_2 \in L(D_2) \} = L(D_1)$. Let $U \subset \mathcal{S}$ be an open subset, we denote $\Lrr(D)(U)$ the set of all function $f \in L(D)$ restricted to $U$. For example, the Riemann--Roch space $L(D) = \Lrr(D)(\mathcal{S})$ is the set of all global sections associated to the divisor $D$.

\subsubsection{Intersection number}
Let $\overline{\F}_q$ be the algebraic closure of $\F_q$ and consider the base field extension $\overline{\mathcal{S}} \coloneq \mathcal{S} \times_{\F_q} \overline{\F}_q$ of $\mathcal{S}$. Let $C$ and $D$ be two curves on $\overline{\mathcal{S}}$. We say that $C$ and $D$ meet \emph{transversally} at a point $P \in \overline{\mathcal{S}}$ if $P \in C \cap D$ and if the local equations of $C$ and $D$ at $P$ generate the maximal ideal $\mathrm{m}_P$ of the ring $\mathcal{O}_{\overline{\mathcal{S}},P}$ of all the rational functions on $\overline{\mathcal{S}}$ that are regular in a neighbourhood of $P$. The curves $C$ and $D$ are said to be \emph{in general position} if they meet transversally at every point $P \in C \cap D$, and we define their \emph{intersection number} to be $C.D \coloneqq \# C \cap D$. This number only depends on the linear equivalence classes of $C$ and $D$, and can be extended by symmetry and bilinearity to all couples of divisors. For example, each divisor $D$ has a \emph{self--intersection number} $D^2 \coloneqq D.D$.

\subsection{Description of $\mathcal{H}_e$}

For the rest of this paper, we set $e \in \mathbb{N}$. The Hirzebruch surface $\mathcal{H}_e$ over $\F_q$ is the unique rational ruled surface over $\F_q$ with invariant $e$ defined in \cite[V.2.13]{Hart13}. The surface $\mathcal{H}_e$ is a ruled surface over $\mathbb{P}^1$, meaning that there exists a morphism $\pi_e: \mathcal{H}_e \rightarrow \mathbb{P}^1$, and its Picard group is generated by a section $S_e$ and a fibre $F_e$ of $\pi_e$, with the intersection numbers
\[ F_e^2 = 0 \, , \, F_e.S_e = 1 \text{ and } S_e^2=-e \, .
\]
The open subset $U \coloneqq \mathcal{H}_e \setminus (S_e \cup F_e)$ is isomorphic to $\mathbb{A}^2$, giving to $\mathcal{H}_e$ its rationality.
Several explicit constructions of $\mathcal{H}_e$ are given in the literature, for instance as a rational ruled surface \cite[V.2]{Hart13} or as a toric variety  \cite{cox2024toric}. All these constructions are equivalent by the uniqueness of such a surface. 

We describe below an explicit construction of the surfaces $\mathcal{H}_e$, by a sequence of blow-up and blow-down, as given for instance in \cite[1.6.3]{Couv20} or in \cite[V.5]{Hart13} for any ruled surfaces. 
\begin{itemize}
 \item The surface $\mathcal{H}_0$ is the product $\mathbb{P}^1 \times \mathbb{P}^1$, the divisors $S_0$ and $F_0$ are two non-equivalent lines that intersect at a single point $Q_0$. We have the intersection numbers $S_0^2 = F_0^2 = 0$ and $S_0.F_0 = 1$. See Figure \ref{fig:Hirzebruch0et1}.
\item $\mathcal{H}_1$ is the blow-up of $\mathbb{P}^2$ at a point $P \in \mathbb{P}^2$, and we denote $S_1$ the exceptional divisor, $F_1$ the strict transform of a line $L$ of $\mathbb{P}^2$ containing $P$, and $Q_1$ the point at the intersection between $S_1$ and $F_1$. Since $S_1$ is the exceptional divisor of a blow-up, we have $S_1^2=-1$; and since $F_1$ intersect with $S_1$, then $F_1^2=L^2 - 1 = 0$. See Figure \ref{fig:Hirzebruch0et1}.
\item To construct $\mathcal{H}_{e+1}$ from $\mathcal{H}_e$, denote $Q_e$ the point at the intersection of $S_e$ and $F_e$. The blow-up of $\mathcal{H}_e$  at $Q_e$ is a surface on which the strict transform of $F_e$ is a projective line with self-intersection number $-1$, and $\mathcal{H}_{e+1}$ is obtained by blowing it down, using Castelnuovo's criterion \cite[V.5]{Hart13}. See Figure \ref{fig:HirzebruchConstruction} for an illustration, which also describes how to construct $\mathcal{H}_{e}$ from $\mathcal{H}_{e+1}$.
\end{itemize}

\begin{figure}
\begin{center}
\begin{tikzpicture}
\begin{scope}[scale=0.7]

\begin{scope}[xshift = -8cm]
\foreach \i in {0,...,4}
	\draw [dashed] (0,\i)--(5,\i);
\foreach \i in {1,...,5}
	\draw [dashed] (\i ,0)--(\i ,5);
\draw [color = red, thick] (0,0) -- (0,5);
\draw [color = blue, thick] (0,5) -- (5,5);
\draw (0,5) node {$\bullet$};
\draw (0,5) node [above left] {$Q_0$};
\draw (0,2.5) node [left] {$\textcolor{red}{S_0}$};
\draw (2.5,5) node [above] {$\textcolor{blue}{F_0}$};
\draw (2.5,-1) node {$\mathcal{H}_0 = \mathbb{P}^1 \times \mathbb{P}^1$};
\draw (6,2) node {;};
\end{scope}

\foreach \i in {0,...,4}
	\draw [dashed] (0,\i)--(5,\i);
\foreach \i in {1,...,5}
	\draw [dashed] (\i ,0)--(3 ,5);
\draw [color = red, thick] (0,0) -- (0,5);
\draw [color = blue, thick] (0,5) -- (5,5);
\draw (0,5) node {$\bullet$};
\draw (0,5) node [above left] {$Q_1$};
\draw (0,0) node [below] {$\textcolor{red}{S_1} = \pi^{-1}(P)$};
\draw (2.5,5) node [above] {$\textcolor{blue}{F_1}$};
\draw (2.5,-1) node {$\mathcal{H}_1$};
\draw [ ->] (5.5,2.5) -- (7,2.5);
\draw (6.25,2.5) node [above] {$\pi$};
\draw (9.5,-1) node {$\mathbb{P}^2$};

\begin{scope}[xshift = 7.5cm, yshift= 2.5cm, scale=1.2]
\draw [color = blue, thick] (0,0)--(3,2);
\draw (1.5,1) node [above left] {$\textcolor{blue}{L}$};
\foreach \i in{-2,-1.2,...,2}
	\draw [dashed] (3,\i) -- (0,0);
\foreach \i in {1,...,5}
	\draw [dashed] (0.6*\i, -0.4*\i) -- (3,2);
\draw (0,0) node {$\textcolor{red}{\bullet}$};
\draw (0,0) node [above] {$P$};
\end{scope}

\end{scope}
\end{tikzpicture}
\end{center}
\caption{The surfaces $\mathcal{H}_0$ and $\mathcal{H}_1$.}
\label{fig:Hirzebruch0et1}
\end{figure}
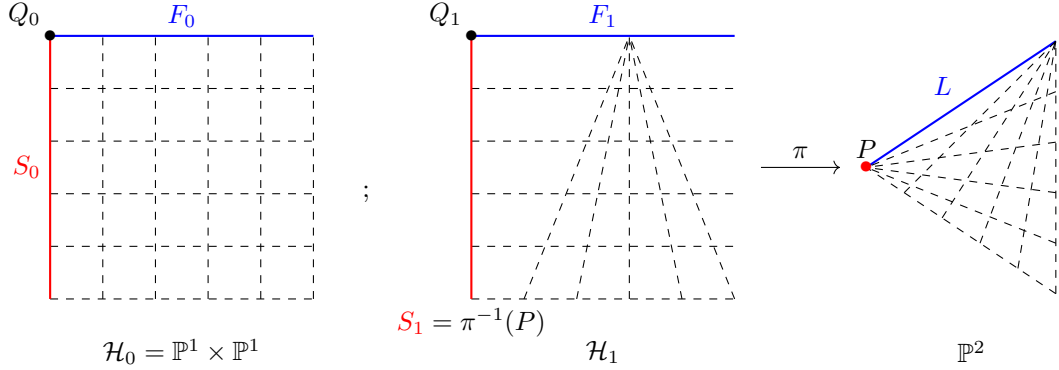 
\begin{figure}
\begin{center}
\begin{tikzpicture}
\begin{scope}[scale=0.7]

\begin{scope}[xshift=-7.5cm]
\draw [color = red, thick] (0,0) -- (0,5);
\draw [color = blue, thick] (0,5) -- (5,5);

\foreach \i in {0,...,4}
	\draw [dashed] (0,\i)--(5,\i);
\foreach \i in {1,...,5}{
	\draw [dashed] (\i ,4)--(\i ,0);
	\draw [dashed] (\i,4) ..controls +(0,0.5) and +(-0,-0.5).. (3,5);
	\draw [dashed] (\i,6) ..controls +(0,-0.5) and +(-0,0.5).. (3,5);
}
\fill [color=white] (0.5,5.5) -- (5.1,5.5) -- (5.1,6.1) -- (0.5,6.1) -- cycle;
\draw (3,5) node {$\bullet$};
\draw (0,5) node {$\textcolor{olive}{\bullet}$};
\draw (0,5) node [above left] {$Q_e$};
\draw (0,2.5) node [left] {$\textcolor{red}{S_e}$};
\draw (1,5) node [above] {$\textcolor{blue}{F_e}$};
\draw (2.5,-1) node {$\mathcal{H}_e$};

\draw [<-] (5.5,2.5) -- (7.,2.5);
\draw (6.25,2.5) node [above] {$\varphi$};
\end{scope}

\draw [color = red, thick] (0,0) -- (0,5);

\foreach \i in {0,...,4}
	\draw [dashed] (0,\i)--(5,\i);
\foreach \i in {1,...,5}{
	\draw [dashed] (\i ,0)--(\i ,4);
}

\draw [color = blue, thick] (5,5) -- (1.25,6.25);

\foreach \i in {1,...,5}{
	\draw [dashed] (\i,4) ..controls +(0,0.75) and +(-0.25,-0.5).. (3.25,5.55);
}

\begin{scope}[xshift=3.25cm, yshift=5.55cm]
\begin{scope}[rotate=-20]

\foreach \i in {1,...,5}
	\draw [dashed] (\i-3,1) ..controls +(0,-0.5) and +(-0,0.5).. (0,0);

\fill [color=white] (-2.05,0.5) -- (2.6,0.5) -- (2.6,1.05) -- (-2.05,1.05) -- cycle;
\end{scope}
\end{scope}
\draw [color = olive, thick] (0,5) -- (2.5,7.5);
\draw (2.5,7.5) node [right] {$\varphi^{-1}(Q_e)$};
\draw (0,5) node {$\bullet$};
\draw (3.25,5.55) node {$\bullet$};

\draw (6.25,2.5) node [above] {$\pi$};
\draw [->] (5.5,2.5) -- (7.,2.5);

\begin{scope}[xshift=7.5cm]
\draw [color = red, thick] (0,0) -- (0,5);
\draw [color = olive, thick] (0,5) -- (5,5);

\foreach \i in {0,...,4}
	\draw [dashed] (0,\i)--(5,\i);
\foreach \i in {1,...,5}{
	\draw [dashed] (\i ,4)--(\i ,0);
	\draw [dashed] (\i,4) ..controls +(0,0.75) and +(-0,-1).. (3,5);
	\draw [dashed] (\i,6) ..controls +(0,-0.5) and +(-0,1).. (3,5);
}
\fill [color=white] (0.5,5.65) -- (5.1,5.65) -- (5.1,6.1) -- (0.5,6.1) -- cycle;
\draw (3,5) node {$\textcolor{blue}{\bullet}$};
\draw (0,5) node {$\bullet$};
\draw (0,5) node [above left] {$Q_{e+1}$};
\draw (0,1.5) node [left] {$\textcolor{red}{S_{e+1}}$};
\draw (1,5) node [above] {$\textcolor{olive}{F_{e+1}}$};
\draw (2.5,-1) node {$\mathcal{H}_{e+1}$};

\end{scope}

\end{scope}
\end{tikzpicture}
\end{center}
\caption{Construction of $\mathcal{H}_{e+1}$ from $\mathcal{H}_e$.}
\label{fig:HirzebruchConstruction}
\end{figure}
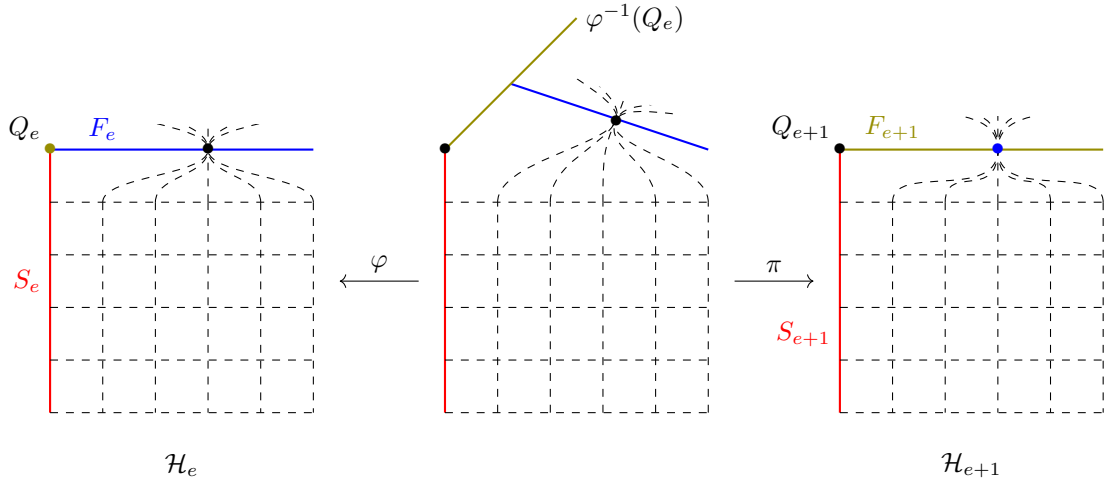

\begin{remark}
From the construction above we see that the only value of $e$ for which $\mathcal{H}_e$ contains a projective line with self intersecting number $-1$ is $e=1$, making $\mathcal{H}_1$ the only non-minimal Hirzebruch surface.
\end{remark}

This construction of Hirzebruch surfaces by composition of blow-ups and blow-downs gives a birational map between $\mathbb{P}^2$ and $\mathcal{H}_e$, with an isomorphism between the affine varieties $\mathbb{A}^2 = \mathbb{P}^2 \setminus L$ and $\mathcal{H}_e \setminus (S_e \cup F_e)$. By considering their intersection with the open $\mathbb{A}^2$, each line of $\mathbb{P}^2$ distinct from $L$ that contains $P$ gives a fibral divisor $F_e' \in \mathcal{H}_e$, and each line that does not contain $P$ gives another section $\sigma_e$ (see Figure \ref{fig:HirzebruchDiviseurs}).

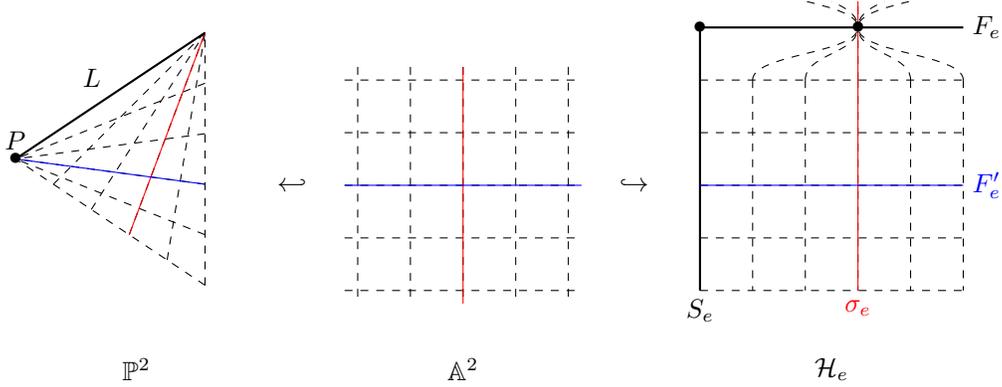
\begin{figure}
\begin{center}
\begin{tikzpicture}
\begin{scope}[scale=0.7]
\begin{scope}[xshift=-7.5cm]

\begin{scope}[xshift = -5.5cm, yshift= 2.5cm, scale=1.2] 
\draw (1.5,1) node [above left] {$L$};
\foreach \i in{-2,-1.2,...,2}
	\draw [dashed] (3,\i) -- (0,0);
\foreach \i in {1,...,5}
	\draw [dashed] (0.6*\i, -0.4*\i) -- (3,2);
\draw [color = red] (1.8,-1.2)--(3,2);
\draw [color = blue] (3,-0.4) -- (0,0);
\draw (0,0) node {$\bullet$};
\draw (0,0) node [above] {$P$};
\draw [thick] (3,2) -- (0,0);
\end{scope}
\draw (-3.2, -1.5) node {$\mathbb{P}^2$};

\foreach \i in {0,...,4}
	\draw [dashed] (0.75,\i)--(5.25,\i);
\foreach \i in {1,...,5}{
	\draw [dashed] (\i ,4.25)--(\i ,-0.25);
}

\draw [color = red] (3,4.25) -- (3,-0.25);
\draw [color = blue] (0.75,2) -- (5.25,2);

\draw (3,-1.5) node {$\mathbb{A}^2$};

\draw (6.25,2) node {$\hookrightarrow$};
\draw (-.25,2) node {$\hookleftarrow$};
\end{scope}

\draw [thick] (0,0) -- (0,5);

\foreach \i in {0,...,4}
	\draw [dashed] (0,\i)--(5,\i);
\foreach \i in {1,...,5}{
	\draw [dashed] (\i ,4)--(\i ,0);
	\draw [dashed] (\i,4) ..controls +(0,0.5) and +(-0,-0.5).. (3,5);
	\draw [dashed] (\i,6) ..controls +(0,-0.5) and +(-0,0.5).. (3,5);
}
\draw [color = red] (3,5.5) -- (3,0);
\draw [color = blue] (0,2) -- (5,2);

\fill [color=white] (0.5,5.5) -- (5.1,5.5) -- (5.1,6.1) -- (0.5,6.1) -- cycle;
\draw [thick] (5,5) -- (0,5);

\draw (3,5) node {$\bullet$};
\draw (0,5) node {$\bullet$};
\draw (0,0) node [below] {$S_e$};
\draw (5,5) node [right] {$\textcolor{black}{F_e}$};
\draw (3,0) node [below] {$\textcolor{red}{\sigma_e}$};
\draw (5,2) node [right] {$\textcolor{blue}{F'_e}$};
\draw (2.5,-1.5) node {$\mathcal{H}_e$};

\end{scope}
\end{tikzpicture}
\end{center}
\caption{The fibres $F'_e$ and $F_e$, and the sections $S_e$ and $\sigma_e \simeq S_e + eF_e$.}
\label{fig:HirzebruchDiviseurs}
\end{figure}

The curve $\sigma_e$ does not meet $S_e$ but intersect $F_e$ transversally, while $F'_e$ intersect $S_e$ transversally without meeting $F_e$. Using the intersection numbers of $F_e$ and $S_e$ (which, we recall, generate the Picard group of $\mathcal{H}_e$), we have the linear equivalences 

\begin{equation*}
\sigma_e \simeq S_e + eF_e \;\text{ and }\; F'_e \simeq F_e \, .
\end{equation*}
In \cite{nardi2019algebraic}, the elements of the Picard group are given with $\sigma_e$ and $F_e$ as generators, and one must keep in mind that for all $(a,b) \in \mathbb{Z}^2$, $aS_e + bF_e \simeq a \sigma_e + (b-ea)F_e$.

A property that distinguishes $\sigma_e$ from $S_e$ when $e > 0$ is its self-intersecting number $\sigma_e^2 = e \geq 0$. This number comes from the successive blow-downs in our construction of $\mathcal{H}_e$ from $\mathbb{P}^2$, where $\sigma_e$ meets other sections at a point of $F_e$ distinct from $Q_e$, with multiplicity $e$ (see Figure \ref{fig:HirzebruchConstruction}).

\subsection{Riemann--Roch spaces of $\mathcal{H}_e$}\label{ss:RR}

To end this section, we describe the Riemann--Roch space associated to a divisor on a Hirzebruch surface.

\begin{notation}
In what follows, we denote by $\F_q[X,Y]^h_{k}$ the $\F_q$-vector space of all homogeneous polynomials of $\F_q[X,Y]$ of degree $k \in \mathbb{N}$.
\end{notation}

\begin{proposition}\label{prop:RRHirz}
Let $a,b \in \mathbb{N}$. The Riemann--Roch space $L(aS_e+bF_e)$ is isomorphic to the $\F_q$--vector space denoted $\Lis (aS_e + bF_e)$ generated by all monomials $M \in \F_q[X_1,X_2,T_1,T_2]$ of the form $M=X_1^{d_1}X_2^{d_2}T_1^{c_1}T_2^{c_2}$, with
\begin{equation*}\label{MonomHirz}
\tag{$\bigstar_{a,b}$}
 \left\{
\begin{array}{c}
b-ea = c_1 + c_2 - ed_1,\\
a = d_1 + d_2
\end{array}.
\right . \,
\end{equation*}
The space $\Lis(aS_e+bF_e)$ is explicitely given by

\begin{equation*}\label{RRHirz}
\tag{$L_{a,b}$}
\Lis (aS_e + bF_e) = \sum_{d_1 = 0}^a \left[ \langle ( X_1^{d_1} X_2^{a-d_1} ) \rangle \otimes  \F_q[T_1,T_2]^h_{b-ea+ed_1} \right] \, .
\end{equation*}
\end{proposition}

\begin{proof}
The system of equations \eqref{MonomHirz} for a basis of $\Lis(aS_e + bF_e)$ is described in \cite{nardi2019algebraic} using the interpretation of $\mathcal{H}_e$ as a toric variety.

We fix $d_1 \in \llbracket 0,a \rrbracket$. For all monomials $M=X_1^{d_1}X_2^{d_2}T_1^{c_1}T_2^{c_2}$ verifying Equation \eqref{MonomHirz}, we have $d_2 = a - d_1$ and $c_1 + c_2 = b-ea + ed_1$. These monomials form a basis of $$(X_1^{d_1}X_2^{a-d_1})\F_q[T_1,T_2]^h_{b-ea+ed_1} = \langle  ( X_1^{d_1} X_2^{a-d_1} ) \rangle \otimes  \F_q[T_1,T_2]^h_{b-ea+ed_1}.$$ By summing over all values of $d_1 \in \llbracket 0,a \rrbracket$, we obtain the equality \eqref{RRHirz}.
\end{proof}

The equations for the fibers and the sections $S_e$ and $\sigma_e$ can be expressed using $T_i$ and $X_i$, $i \in \{1,2\}$:

\begin{equation*}
F_e = \{T_1 = 0 \} \,, \, F'_e = \{T_2 = 0 \} \, , \, S_e = \{X_1 = 0 \} \,, \, \sigma_e = \{X_2 = 0 \}.
\end{equation*}

The set $\Lis(aS_e + bF_e)$ in Equation \eqref{RRHirz} of Proposition \ref{prop:RRHirz} is the set of all regular functions on $\mathcal{H}_e \setminus (F_e \cup S_e) = \{X_1T_1 \neq 0 \}$ of \emph{bidegree} $(a,b-ea)$. Since $\mathcal{H}_e$ and $\mathbb{P}^2$ are birational, they both share the same function field $\F_q(X,Y)$. However, we can also write the set of invertible functions on $\mathcal{H}_e$ as the set of rational fractions of \emph{bidegree} zero:
\[
\F_q(\mathcal{H}_e) = \left\{ \frac{f_1}{f_2} \,\bigg\rvert \, f_1,f_2 \in \Lis(aS_e + bF_e) \setminus \{ 0 \} , (a,b) \in \mathbb{N}^2 \right\} \cup \{0 \} \, .
\]
Thus, if one wants to interpret the Riemann--Roch space associated to $aS_e+bF_e$ as a subset of $\F_q(\mathcal{H}_e)$, an isomorphism is given by
\[
L(aS_e + bF_e) = \left( \frac{1}{X_1^aT_1^b} \right) L_*(aS_e + bF_e) \subset \F_q(\mathcal{H}_e) \, ,
\]
with the space $L_*(aS_e + bF_e)$ explicitly described in Proposition \ref{prop:RRHirz}. We generalize this formula for any divisor of $\mathcal{H}_e$ in the next proposition.

\begin{proposition}\label{prop:RRgeneral}
Let $G \in \Div(\mathcal{H}_e)$ and write $G + (f) = aS_e + bF_e$ where $f \in \F_q(\mathcal{H}_e)$ and $(a,b) \in \mathbb{Z}^2$. Its Riemann--Roch space is 
\[
L( G) = \left( \frac{f}{X_1^aT_1^{b}} \right)  L_*(aS_e + bF_e) \, ,
\]
where $L_*(aS_e+bF_e)$ is given in Equation \eqref{RRHirz} of Proposition \ref{prop:RRHirz}.
\end{proposition}

\begin{proof}
The equality $G + (f) = aS_e + bF_e$ comes from the generating family for the Picard group $\Pic(\mathcal{H}_e) = \mathbb{Z}S_e + \mathbb{Z}F_e$. Let $h \in \F_q(\mathcal{H}_e) \setminus \{ 0 \}$, then the following chain of equivalences holds
\begin{align*}
h \in L( G) &\Longleftrightarrow (h) \geq -G = -aS_e - bF_e + (f) \\
& \Longleftrightarrow \left( \frac{h}{f} \right) \geq -aS_e - bF_e \\
& \Longleftrightarrow \frac{h}{f} \in  \left( \frac{1}{X_1^aT_1^{b}} \right) L_*(aS_e + bF_e) \, ,
\end{align*}
concluding the proof.
\end{proof}

As described in \cite[2.2]{nardi2019algebraic}, associating to all monomial $M=X_1^{d_1}X_2^{d_2}T_1^{c_1}T_2^{c_2}$ verifying Equation \eqref{MonomHirz}  the point of coordinates $(d_2,c_2) \in \mathbb{N}^2$, a basis for $\Lis(aS_e + bF_e)$ is represented by all the points of $\mathcal{P}_{a,b} \cap \mathbb{N}^2$, where we define the polygon \[\mathcal{P}_{a,b} = \{ (x,y) \in \mathbb{R}^2 \, | \, 0 \leq x \leq a \, \text{and} \, 0 \leq y \leq b - e x\} \, . \]
See \cite[Fig. 2]{nardi2019algebraic} for an illustration of the different possible shapes for $\mathcal{P}_{a,b}$, namely a trapezoid when $b-ea \geq 0$ (a rectangle when $e=0$) and a triangle otherwise.

\section{Algebraic Geometry codes on $\mathcal{H}_e$ and their dual distance}\label{section:2}

In this section, we first recall the notion of linear and algebraic geometry (AG) codes , secondly we give in Definition \ref{def:codehirz} the construction of the AG code $C_e(a,b)$ on the Hirzebruch surface $\mathcal{H}_e$, parametrized by the divisor $aS_e + bF_e$. We then give an explicit form of $C_e(a,b)$ in Proposition \ref{prop: Code Hirzebruch}, using Reed--Solomon codes. Thanks to this explicit form, we compute the dimension $k_e(a,b)$ and the dual minimum distance $d_e(a,b)^\perp$ of $C_e(a,b)$, in Proposition \ref{prop:paramcode} and Theorem \ref{prop:dualdistance}, respectively. 
\subsection{Context}

\subsubsection{Codes}

A \emph{linear error-correcting code}, or more simply a \emph{code}, over $\F_q$ is an $\F_q$-vector space $C \in \F_q^n$, $n \in \mathbb{N}$. We call $n$ the \emph{length} of $C$ and $\dim_{\F_q}(C)$ its \emph{dimension}. The elements of $C$ are called \emph{words} or \emph{codewords}. The \emph{Hamming weight} $\wt(c)$ of $c \in C$ is the number of non-zero entries of $c$, and the \emph{minimum distance} of $C$ is $d(C) \coloneqq \min \{\wt(c) \, | \, c \in C \setminus \{ 0 \} \}$. A code over $\F_q$ of length $n$, dimension $k$ and minimum distance $d$ is called a $[n,k,d]_q$\emph{--code}.

Let $C$ be a code over $\F_q$ of length $n$, and take $x = (x_1, \dots, x_n)$ and $x'=(x'_1, \dots, x'_n) \in \F_q^n$. With the Euclidean scalar product of $x$ and $x'$ defined as 
\[ x \cdot x' \coloneqq \sum_{i=1}^n x_ix'_i \, ,
\]
we define the \emph{dual} of $C$ to be its orthogonal space
\[C^\perp = \{ x \in \F_q \, | \, \forall c \in C, x \cdot c = 0 \}.
\]
The set $C^\perp$ is a $[n,n-\dim(C),d(C^\perp)]_q$--code with $C = (C^\perp)^\perp$, and $d(C^\perp)$ is called the \emph{dual minimum distance} or \emph{dual distance} of $C$. 

Let $C_1$ and $C_2$ be two codes over $\F_q$, of length $n_1$ and $n_2$ respectively. We define the tensor product $C_1 \otimes C_2$ to be the set of all $n_2 \times n_1$ matrices whose rows are elements of $C_1$ and columns are elements of $C_2$, which is coherent with the tensor product of $C_1$ and $C_2$ as vector-spaces. It is a $[n_1n_2, k_1k_2, d_1d_2]_q$--code. Indeed, by taking $c_1$ and $c_2$ minimally-weighted words of $C_1 \setminus \{0 \}$ and $C_2 \setminus \{0 \}$ respectively, the pure tensor $c_1 \otimes c_2$ has minimum distance $d_1d_2$ and so $d(C_1 \otimes C_2) \leq d_1d_2$. For the lower bound, observe that a non-zero word $c \in C_1 \otimes C_2$ has at least $d_2$ non-zero rows, and each non-zero row has minimum distance $d_1$.

Finally, let $I \subsetneq \llbracket 1,n \rrbracket$ and $C \in \F_q^n$. The \emph{puncturing at} $I$ of $C$ is the image of the projection
\[ \pi_I : \left\{
\begin{array}{c c c}
C &\rightarrow & \F_q^{k} \\
(c_1, \dots, c_n) &\mapsto &(c_{i_1}, \dots, c_{i_k})
\end{array}
\right. \, ,
\]
where $k = n - \# I$ and $\{ i_1, \dots, i_k \} = \llbracket 1,n \rrbracket \setminus I$.

\subsubsection{AG codes on a surface}
Let $\mathcal{S}$ be a smooth projective and geometrically integral surface over $\mathbb{F}_q$. Define $\Delta \subset \mathcal{S}(\F_q)$ a non-empty set of $\F_q$-rational points of $\mathcal{S}$ of cardinality $n$, and $G \in \Div(\mathcal{S})$ such that no point $P \in \Delta$ is in the support of $G$. We define the \emph{Algebraic Geometry code} (\emph{AG code} for short) on $\mathcal{S}$, parametrized by $\Delta$ and $G$, to be the image of the evaluation morphism

\[
\ev_\Delta : \left\{ \begin{array}{c c c} L(G) &\rightarrow &\F_q^{n} \\ f &\mapsto &(f(P))_{P \in \Delta} \end{array} \right. \, .
\]

\begin{remark}\label{remarque chiante} Let $G = G_0 - G_\infty$, where $G_0$ and $G_\infty$ are positive divisors with no common prime components. Having a point $P \in \Delta$ outside of the support of $G$, and of $G_0$ in particular, ensures that the evaluation at $P$ of $f \in L(G) \subset L(G_0)$ is well-defined. One may circumvent this condition on $G$ by choosing for every point $P \in \Delta$ a generator $t_P$ of the local ring $\mathcal{O}(G)_P$, and by evaluating $f_P(P) \in \F_q^n$ for each $f=f_Pt_P \in L(G)$.
\end{remark}

\subsection{Evaluation morphism}

We define here an AG code $C_e(a,b)$, by evaluating each function of the space $\Lis(a S_e + b F_e)$ at every $\F_q$-rational point of $\mathcal{H}_e$.

\begin{notation}\label{nota: P1}
As is the custom, we fix the following set of representatives for the points of $\mathbb{P}^1(\F_q)$:
\[ \mathbb{P}^1(\F_q) = \{ (0,1) \} \cup \{ (1,\alpha) \, | \, \alpha \in \F_q \}  \, . \]
We will call the point $(0,1)$ "the point at infinity".
\end{notation}

According to \cite{nardi2019algebraic}, the $\F_q$-rational points of $\mathcal{H}_e$ correspond to those of the "grid" $\mathbb{P}^1 \times \mathbb{P}^1$. Using our choice of representatives (\ref{nota: P1}), each point $P \in \mathcal{H}_e(\F_q) \simeq \mathbb{P}^1(\F_q) \times \mathbb{P}^1(\F_q)$ has one of the following forms
\begin{itemize}
\item $(1, \alpha, 1, \beta) \; \text{with} \, (\alpha, \beta) \in \F_q^2 $, when $P \in \mathcal{H}_e \setminus (F_e \cup F_e)$;
\item $(1, \alpha, 0, 1) \; \text{with} \, \alpha \in \F_q$, when $P \in F_e \setminus \{ Q_e \}$;
\item $(0,1,1,\beta) \; \text{with} \, \beta \in \F_q$, when $P \in S_e \setminus \{ Q_e \}$;
\item $(0,1,0,1)$, when $P= Q_e$.
\end{itemize}
Note that $\# \mathcal{H}_e (\F_q) = (q+1)^2$.
 
\begin{convention}
For ease of notation, we will use the convention 
\[
0^0 \coloneqq 1 \, ,
\]
so that for all $k \in \mathbb{N}$, $0^k = \left \{ \begin{array}{r} 0 \text{ if } k \neq 0 \\ 1 \text{ if } k = 0 \end{array} \right. \, .$
\end{convention}

\begin{definition}[The code $C_e(a,b)$]\label{def:codehirz}
 The evaluation of a monomial $M = X_1^{d_1}X_2^{d_2}T_1^{c_1}T_2^{c_2}$ at a point $P = (x_1,x_2,t_1,t_2) \in \mathcal{H}_e(\F_q)$ is defined by $M(P) = x_1^{d_1}x_2^{d_2}t_1^{c_1}t_2^{c_2}$. We extend it by linearity to define an evaluation $f(P)$ for all $f \in \Lis(aS_e + bF_e)$ (defined in Proposition \ref{prop:RRHirz}). Let $a,b \in \mathbb{N}$, and let $\ev_{\mathcal{H}_e}$ be the evaluation morphism

\[
\ev_{\mathcal{H}_e}: \left\{
\begin{array}{r c l}
\Lis(aS_e + bF_e) & \rightarrow & \F_q^{(q+1)^2} \\
f & \mapsto & (f(P))_{P \in \mathcal{H}_e(\F_q)}
\end{array}
\right . \, .
\]
The AG code $C_e(a,b)$ defined on $\mathcal{H}_e$ and parametrized by $(a,b) \in \mathbb{N}^2$ is the image
\[
C_e(a,b) \coloneqq \Image(\ev_{\mathcal{H}_e}) \subset \F_q^{(q+1)^2} \, .
\]
\end{definition}

Defining the evaluation morphism on $\Lis(aS_e+bF_e)$ instead of $L(aS_e+bF_e)$ corresponds to our choice of a generator $t_P$ at each point $P$ that is in the support of $aS_e + bF_e$, as discussed in Remark \ref{remarque chiante}. We will see in Section \ref{section:equivalence} that $C_e(a,b)$ is equivalent to every AG code on the points of $\mathcal{H}_e(\F_q)$ parametrized by a divisor $G \simeq aS_e+bF_e$ with no $\F_q$-rational point in its support.

A generating set of the $\F_q$-vector space $C_e(a,b)$ is given by the images of the monomials $M=T_1^{c_1}T_2^{c_2}X_1^{d_1}X_2^{d_2}$ via $\ev_{\mathcal{H}_e}$, which we represent by square matrices with $(q+1)^2$ coordinates. We organize the index of each coordinate to be coherent with our representation of $\mathcal{H}_e$ (see Figure \ref{fig:HirzebruchCode}), with the first column being formed by the evaluations of $M$ at the $\F_q$-rational points of $S_e = \{X_1 = 0 \}$ (all coordinates are zero if and only if $d_1 \neq 0$) and the first row being the ones at $F_e = \{T_1 = 0 \}$ (which are zero if and only if $c_1 \neq 0$).

\begin{figure}[H]
\begin{center}
\begin{tikzpicture}
\begin{scope}[scale=0.65]

\foreach \i in {0,...,2}
	\draw [dashed] (0,\i)--(3,\i);
\foreach \i in {1,...,3}{
	\draw [dashed] (\i ,2)--(\i ,0);
	\draw [dashed] (\i,2) ..controls +(0,0.5) and +(-0,-0.5).. (2,3);
}

\draw (1.5,-1) node {$\mathcal{H}_e(\F_q)$};

\draw (3,3) -- (0,3);

\draw (0,0) -- (0,3);

\foreach \i in {0,...,3}{
	\foreach \j in {0,...,3}
		\draw (\i,\j) node {$\bullet$};
} 

\draw (4,1.5) node {$;$};
\draw (5,1.5) node [right] {$\ev_{\mathcal{H}_e}(M) = \begin{pmatrix}
0^{c_1 + d_1} & 0^{c_1}\alpha_1^{d_2} & \cdots & 0^{c_1}\alpha_q^{d_2}\\
\alpha_1^{c_2}0^{d_1} & \alpha_1^{c_2}\alpha_1^{d_2} & \cdots & \alpha_1^{c_2}\alpha_q^{d_2} \\
\vdots & \vdots &  & \vdots \\
\alpha_q^{c_2}0^{d_1} & \alpha_q^{c_2}\alpha_1^{d_2} & \cdots & \alpha_q^{c_2}\alpha_q^{d_2}
\end{pmatrix}$};
\end{scope}
\end{tikzpicture}
\end{center}
\caption{Codeword associated to the monomial $M=T_1^{c_1}T_2^{c_2}X_1^{d_1}X_2^{d_2}$.}
\label{fig:HirzebruchCode}
\end{figure}

\begin{notation}
Let $a,b \in \mathbb{N}$. Using the basis $(\sigma_e, F_e)$ of $\Pic(\mathcal{H}_e)$, the code $C_e(a,b)$ is parametrized by the divisor $a\sigma_e + (b-ea) F_e \simeq aS_e + bF_e$. We will denote 
\[
C_e(a\sigma_e + (b-ea)F_e) \coloneqq C_e(a,b) \, .
\]
\end{notation}

The construction of the Riemann--Roch space $L(aS_e + bF_e)$ for all integers $a$ and $b$ allows us to give an explicit description of $C_e(a,b)$ using projective Reed--Solomon codes, that we define below.

\begin{definition}[Reed--Solomon codes]
Denote $\F_q = \{ \alpha_1, \dots, \alpha_q \}$. The \emph{projective Reed--Solomon code} of dimension $k \in \llbracket 0, q+1 \rrbracket$ on the alphabet of size $q$, denoted $\PRS_q(k)$, is the image of
\[
\ev_{\mathbb{P}^1(\F_q)} : \left\{
\begin{array}{r c l}
\F_q [X,Y]^h_{k-1} &\rightarrow &\F_q^{q+1} \\
P(X,Y) & \mapsto & (P(0,1), P(1,\alpha_1), \dots, P(1, \alpha_q)) 
\end{array}
\right. .
\]
The canonical basis of $\PRS_q(k)$ is  $((0^{k-1-d}, \alpha_1^d, \dots, \alpha_q^d))_{d \in \llbracket 0, k-1 \rrbracket}$. We extend the definition of $\PRS_q(k)$ for all integers $k \in \mathbb{Z}$ by putting $\PRS_q(k)= \{0 \}$ when $k \leq 0$, and $\PRS_q(k) = \F_q^{q+1}$ when $k \geq q+1$.

The \emph{(affine) Reed--Solomon code} $\RS_q(k)$ is the puncturing of $\PRS_q(k)$ at the first coordinate. Its dimension is $k$ if and only if $k \in \llbracket 0,q \rrbracket$. In that case, the canonical basis of $\RS_q(k)$ is the family $(\alpha^{*d})_{d \in \llbracket 0, k-1 \rrbracket}$, with $\alpha^{*d} \coloneqq (\alpha_1^d, \dots, \alpha_q^d)$.
\end{definition}

\begin{proposition}[{\cite[2.3]{Stic09}}]
Let $k \in \llbracket 0,q \rrbracket$. The minimum distance of $\PRS_q(k)$ and $\RS_q(k)$ are, respectively,
\begin{align*}
d(\PRS_q(k)) &= q+2-k \, ,\\
d(\RS_q(k)) &= q+1 - k \, .
\end{align*}

Their duals are respectively given by
\begin{align*}
\PRS_q(k)^\perp &= \PRS_q(q+1-k) \, ,\\
\RS_q(k)^\perp &= \RS_q(q-k) \, .
\end{align*}

\end{proposition}

We can now describe the codes $C_e(a,b)$ using sums of tensor-products of codes.

\begin{proposition}\label{prop: Code Hirzebruch}
Let $a,b \in \mathbb{N}$, and let $\F_q = \{ \alpha_1, \dots, \alpha_q \}$. The evaluation code $C_e(a,b)$ on the $\F_q$-rational points of the Hirzebruch Surface $\mathcal{H}_e$ is given by
\[
C_e(a,b) = \sum_{d_1 = 0}^a \left[ \langle ( 0^{d_1}, \alpha_1^{a-d_1}, \dots, \alpha_q^{a-d_1} ) \rangle \otimes  \PRS_q(b-ea+ed_1+1) \right] \, .
\]
\end{proposition}

\begin{proof}
We prove it by "evaluating" the set $\Lis(aS_e+bF_e)$ of Equation \eqref{RRHirz} at every $\F_q$-rational point of $\mathcal{H}_e$. Fix $d_1 \in \llbracket 0, a \rrbracket$ and let $f = X_1^{d_1}X_2^{a-d_1} \widetilde{f}(T_1,T_2) \in \langle ( X_1^{d_1} X_2^{a-d_1} )  \rangle \otimes \F_q[T_1,T_2]^h_{b-ea+ed_1}$. Evaluating $f$ at each $\F_q$-rational point of $\mathcal{H}_e$ gives the following codeword
\begin{align*}
 \ev_{\mathcal{H}_e}(f) &= (x_1^{d_1}x_2^{a-d_1}\widetilde{f}(t_1,t_2))_{(t_1:t_2)\times(x_1:x_2) \in \mathcal{H}_e} \\
			 &=  \ev_{\mathbb{P}^1(\F_q)}(X_1^{d_1}X_2^{a-d_1}) \otimes \ev_{\mathbb{P}^1(\F_q)}(\widetilde{f}(T_1,T_2)) \, .
\end{align*}
Since we chose the system of representatives $\{(0,1)\} \cup \{ (1,\alpha) \, | \, \alpha \in \F_q \}$ for the elements of $\mathbb{P}^1 (\F_q)$, and since $\widetilde{f}$ is homogenuous of degree $b-ea+ed_1$, we get
\[
\ev_{\mathcal{H}_e}(f) \in  \langle ( 0^{d_1}, \alpha_1^{a-d_1}, \dots, \alpha_q^{a-d_1} ) \rangle \otimes  \PRS_q(b-ea+ed_1+1) \, .
\]
Conversely, for every element $c = \ev_{\mathbb{P}^1(\F_q)}(P) \in  \PRS_q(b-ea+ed_1+1)$ where $P$ is an homogenuous polynomial of degree $b-ea+ed_1$, we have
\[
\langle ( 0^{d_1}, \alpha_1^{a-d_1}, \dots, \alpha_q^{a-d_1} ) \rangle \otimes c = \ev_{\mathcal{H}_e} (X_1^{d_1}X_2^{a-d_1} P(T_1,T_2)) \in C_e(a,b) \, .
\] 
We conclude by summing the spaces $ \langle ( 0^{d_1}, \alpha_1^{a-d_1}, \dots, \alpha_q^{a-d_1} ) \rangle \otimes  \PRS_q(b-ea+ed_1+1)$ over all integers $d_1 \in \llbracket 0,a \rrbracket$.
\end{proof}

Despite having an explicit form for $C_e(a,b)$, studying its dual is rendered difficult by the fact that for two integers $0 \leq k_1 \leq k_2 \leq q$, $\PRS_q(k_1)$ is not a subset of $\PRS_q(k_2)$ in general. As a consequence, sums and intersections between projective Reed--Solomon codes do not behave well.

\begin{notation}
Let $C$ be a code over $\F_q$ of length $n$. We denote by $(0,C)$ the concatenation of length $n+1$
\[
(0,C) = \{ (0,c_1,\dots,c_n) \mid (c_1, \dots, c_n) \in C \} \, .
\]
\end{notation}

\begin{lemma}\label{lem:prs}
Let $k_1$ and $k_2$ be two integers such that $1 \leq k_1 < k_2 \leq q$. Then we have
\begin{equation}\label{eq:prs1}
\PRS_q(k_1) \cap \PRS_q(k_2) = (0, \RS_q(k_1 - 1)) \, ,
\end{equation}
and
\begin{equation}\label{eq:prs2}
\PRS_q(k_1) + \PRS_q(k_2) = \langle (1, 0, \dots, 0) \rangle + (0, \RS_q(k_2)) \, .
\end{equation}
\end{lemma}

\begin{proof}
Let $\F_q = \{ \alpha_1, \dots, \alpha_q \}$. For all $k \in \mathbb{N}^*$, $\PRS_q(k) = \langle (1, \alpha_1^{k -1}, \dots, \alpha_q^{k - 1}) \rangle + (0, \RS_q(k-1))$, and $\RS_q(k_1) \subset \RS_q(k_2)$, so
\[ (0,\RS_q(k_1 -1)) \subset \PRS_q(k_1) \cap \PRS_q(k_2) \, .
\]
Since $\dim (\RS_q(k_1 - 1)) = \dim( \PRS_q(k_1) ) -1$, we now need to show that $(1, \alpha_1^{k_1 -1}, \dots, \alpha_q^{k_1 - 1}) \notin \PRS_q(k_2)$. Since $k_1 < k_2 \leq q$, the minimum distance of $\PRS_q(k_2)$ is at least $2$ and $(0, \alpha_1^{k_1 - 1}, \dots, \alpha_q^{k_1-1}) \in \PRS_q(k_2)$. Thus
\[
(1, \alpha_1^{k_1 -1}, \dots, \alpha_q^{k_1 - 1}) = (1,0, \dots,0) + (0, \alpha_1^{k_1 - 1}, \dots, \alpha_q^{k_1-1}) \notin \PRS_q(k_2) \, ,
\]
and we get Equation \eqref{eq:prs1}. Replacing with $1 \leq q + 1 -k_2 < q + 1 -k_1 \leq q$ and taking the duals, we obtain Equation \eqref{eq:prs2}
\end{proof}

\subsection{Parameters of $C_e(a,b)$}

We recall here the parameters of the code $C_e(a,b)$ given in \cite{nardi2019algebraic}. We left out some pathological cases, namely $b-ea < 0$ and $a \geq q$, since those do not give particularly interesting codes.

\begin{proposition}[Parameters of $C_e(a,b)$]\label{prop:paramcode}
We fix $e \geq 2$. Let $(a,b) \in \mathbb{N}^2$ such that $0 \leq a \leq q-1$ and $0 \leq b-ea$, and consider the $\left[(q+1)^2, k_e(a,b), d_e(a,b)\right]_q$-code $C_e(a,b)$. 
We set
\[
s \coloneqq \frac{b-q}{e} \, \text{ and } \,\widetilde{s} \coloneqq \left\{ 
\begin{array}{l} 
\min \{ \lfloor s \rfloor , a \} \text{ if } s \geq 0 \\
-1 \text{ otherwise}
\end{array}
\right. \, .
\]
Then, the dimension of $C_e(a,b)$ is
\[
k_e(a,b) = (\widetilde{s}+1)(q+1) + (a - \widetilde{s}) \left( b + 1 - e \left( \frac{a + \widetilde{s}+1}{2} \right) \right) \, .
\]
Its minimum distance $d_e(a,b)$ is given as follow:
\begin{itemize}
\item if $q > b$, then
\[
d_e(a,b) = (q + \mathds{1}_{a = 0} )(q - b + 1) \, ;
\]
\item if $ b-ea < q \leq b$, then
\[
d_e(a,b) = q - \left\lfloor \frac{b-q}{e} \right\rfloor \, ;
\]
\item and if $q \leq b-ea$, then,
\[
d_e(a,b) = q - a + 1 \, .
\]
\end{itemize}
\end{proposition}

\begin{proof}
Benefiting from the conditions $a \leq q-1$ and $b-ea \geq 0$, we offer here a new proof for the dimension $k_e(a,b)$ of $C_e(a,b)$ using Proposition \ref{prop: Code Hirzebruch}. 

Let $\F_q = \{ \alpha_1, \dots, \alpha_q \}$, and $\alpha^{*k} \coloneqq (\alpha_1^k, \dots, \alpha_q^k) \in \F_q^q$ for all integer $k \in \mathbb{N}$. Since $a \leq q-1 < q$, the family of vectors $(0^{a-d}, \alpha^{*d}) = (0^{a-d}, \alpha_1^d, \dots, \alpha_q^d)$, $d \in \llbracket 0,a \rrbracket$, is a basis of $\PRS_q(a+1)$ and thus linearly independant. Recall that $\forall k \in \llbracket 0,q+1 \rrbracket, \, \dim (\PRS_q(k)) = k$ and $k' \geq q+1 \Leftrightarrow \PRS_q(k') = \F_q^{q+1}$. Moreover, for all $d \in \llbracket 0,a \rrbracket$, $b-ed+1 \geq q+1 \Leftrightarrow d \leq \lfloor s \rfloor$. From Proposition \ref{prop: Code Hirzebruch}, we thus express $C_e(a,b)$ as a direct sum, with the change of variables $d \coloneqq d_2 = a-d_1$:
\[
C_e(a,b) = \left( (0, \RS_q(\widetilde{s}+1)) \otimes \F_q^{q+1} \right) \oplus \bigoplus^a_{d=\widetilde{s}+1} \left[ \langle (0^{a-d}, \alpha^{*d}) \rangle \otimes \PRS(b- ed +1) \right] \, .
\]
The dimension $k_e(a,b)$ of $C_e(a,b)$ follows
\begin{align*}
k_e(a,b) &= (\widetilde{s}+1)(q+1) + \left( \sum_{d = \widetilde{s}+1}^a b - ed + 1  \right) \\
& = (\widetilde{s}+1)(q+1) + (a - \widetilde{s}) \left( b + 1 - e \left( \frac{a + \widetilde{s}+1}{2} \right) \right) \, .
\end{align*}

The minimum distance $d_e(a,b)$ is computed in  {\cite[4.2.3]{nardi2019algebraic}}.
\end{proof}

\begin{corollary}[Kernel of the evaluation morphism]\label{coro}
 With the same notations as in Proposition \ref{prop:paramcode}, consider the evaluation morphism (described in Definition \ref{def:codehirz})

\[
\ev_{\mathcal{H}_e}: \left\{
\begin{array}{r c l}
\Lis(aS_e + bF_e) & \rightarrow & C_e(a,b) \\
f & \mapsto & (f(P))_{P \in \mathcal{H}_e(\F_q)}
\end{array}
\right . \, .
\]
Then, the kernel of $\ev_{\mathcal{H}_e}$ has dimension $(\widetilde{s}+1)\left(b-q-\frac{e\widetilde{s}}{2} \right)$.
\end{corollary}

\begin{proof}
Recall the explicit form of Equation \eqref{RRHirz} from Proposition \ref{prop:RRHirz}:
\[
\Lis(aS_e + bF_e) = \sum_{d = 0}^a \left[ \langle ( X_1^{a-d} X_2^{d} ) \rangle \otimes  \F_q[T_1,T_2]^h_{b-ed} \right] \, .
\]
Denote by $\ell(aS_e+bF_e)$ the dimension of $\Lis(aS_e + bF_e)$, then
$\ell(aS_e+bF_e) = \sum_{d=0}^a (b-ed+1)$. In our proof of Proposition \ref{prop:paramcode} for the dimension $k_e(a,b)$ of $C_e(a,b)$, we saw that $k_e(a,b) = (\widetilde{s}+1)(q+1) + \left( \sum_{d = \widetilde{s}+1}^a b - ed + 1  \right)$. We can now easily compute the dimension of the kernel
\begin{align*}
\dim( \ker (\ev_{\mathcal{H}_e})) &= \ell(aS_e + bF_e) - k_e(a,b) \\
& = \sum_{d=0}^{\widetilde{s}} b - ed -q \\
&= (\widetilde{s}+1)\left(b-q-\frac{e\widetilde{s}}{2} \right) \, .
\end{align*}

\end{proof}

\subsection{Dual distance of $C_e(a,b)$}

We now want to estimate the parameters of the $[n,k_e(a,b)^\perp,d_e(a,b)^\perp]_q$-code $C_e(a,b)^\perp$, within the hypotheses of Proposition \ref{prop:paramcode}. Since we already know that $n = (q+1)^2$ and $k_e(a,b)^\perp = (q+1)^2 - k_e(a,b)$, this subsection aims to give a lower bound for $d_e(a,b)^\perp$. To this end, we first introduce the notion of check-product of codes.

\subsubsection{Check-product of codes}

\begin{definition}[Check-product]\label{def:checkprod}
Let $k$ be a field, and let $E \subset k^{n_1}$ and $F \subset k^{n_1}$ be two $k$-vector spaces. The \emph{check-product} of $E$ and $F$, denoted $E \boxplus_k F$, is the $k$-vector space
\[
E \boxplus_k F \coloneqq E \otimes k^{n_2} + k^{n_1} \otimes F \, .
\]
\end{definition}

The main use of this notation is to express the duals of tensor-products of codes:

\begin{proposition}\label{prop:checkproduct}
Let $C_1 \subset \F_q^{n_1}$ and $C_2 \subset \F_q^{n_2}$ be two codes on $\F_q$. Then
\begin{equation}\label{eq:dualtensor}
(C_1 \otimes C_2)^\perp = C_1^\perp \boxplus_{\F_q} C_2^\perp \, .
\end{equation}
Let $d_1^\perp$ and $d_2^\perp$ be the minimum distance of $C_1^\perp$ and $C_2^\perp$, respectively. Then the dual distance $d^\perp$ of $C_1 \otimes C_2$ is
\begin{equation}\label{eq:distdualtensor}
d^\perp = \min \{ d_1^\perp, d_2^\perp \} \, .
\end{equation} 
\end{proposition}

The dual distance $d^\perp$ was already known in the literature, with a proof given in \cite{cross2024quantum} for example. We report it here since some techniques are re-used for the proof of Lemma \ref{lem:lemmetechnique}. 

\begin{proof}
 Recall that $C_1^\perp \otimes \F_q^{n_2}$ is the set of all matrices of size $n_2 \times n_1$ whose rows are elements of $C_1^\perp$, while  $\mathbb{F}_q^{n_1} \otimes C_2^\perp$ is the set of those whose columns are elements of $C_2^\perp$. The intersection of the two spaces is $C_1^\perp \otimes C_2^\perp$ by definition. It is now easy to see that $C_1^\perp \otimes \F_q^{n_2}$ and $\mathbb{F}_q^{n_1} \otimes C_2^\perp$ are included in $(C_1 \otimes C_2)^\perp$, hence, by computing the dimension on both sides, we show Equality \eqref{eq:dualtensor}.

To obtain the dual distance $d^\perp$, we first show that $d^\perp \geq \min \{ d_1^\perp, d_2^\perp \} $. Let $c \in  C_1^\perp \boxplus_{\F_q} C_2^\perp$, $c \neq 0$, and consider the matrix-products $c \times v \in \F_q^{n_2}$ for all $v \in C_1$. Note that by construction of $c$, $\forall x \in C_1, \, c \times x \in C_2^\perp$. There are two cases:
\begin{itemize}
\item if $\forall v \in C_1, \, c \times v = 0$ then $c \in C_1^\perp \otimes \F_q^{n_2}$ and $\wt(c) \geq d_1^\perp$;
\item if $\exists v \in C_1, \, c \times v \in C_2^\perp \setminus \{0 \}$ then $\wt(c) \geq \wt(c \times v)\geq d_2^\perp$.
\end{itemize}
To prove $d^\perp \leq \min \{ d_1^\perp, d_2^\perp \}$, simply consider the codewords $c_1^\perp \otimes (1, 0, \dots, 0)$ and $(1, 0, \dots, 0) \otimes c_2$ where $c_1$ and $c_2$ are codewords of $C_1^\perp \setminus \{ 0 \}$ and $C_2^\perp \setminus \{ 0 \}$ respectively, of minimal weights.

\end{proof}

\begin{example}[Dual of $C_0(a,b)$]
Let $(a,b) \in \llbracket 1, q-1 \rrbracket^2$, and consider the code $C_0(a,b) = \PRS_q(a+1) \otimes \PRS_q(b+1)$. Its dual is
\[
C_0(a,b)^\perp = \PRS_q(q-a) \boxplus_{\F_q} \PRS_q(q-b) \, ,
\]
and its dual distance is $\min \{ a,b \} +2$.
\end{example}

In the general case, we can precisely describe the set of all codewords of minimum weight, in a check-product of codes. This is the aim of the following lemma.

\begin{lemma}\label{lem:lemmetechnique}
Let $C_1$ and $C_2$ be two codes over $\F_q$ of length $n_1$ and $n_2$, and of minimum distances $d_1$ and $d_2$, respectively. Suppose $2 \leq d_1 < d_2$, and consider the check-product
\[
C \coloneqq C_1 \boxplus_{\F_q} C_2 \, .
\] 
Then, the words $c \in C \setminus \{ 0 \}$ of minimum weight are the $n_2 \times n_1$ matrices with one row being a minimally weighted word of $C_1 \setminus \{ 0 \}$, and the other rows equal to zero.
\end{lemma}

\begin{proof}
Let $c \in C$ be a codeword of weight $d_1$, and denote by $c_1, \dots, c_{n_2}$ its rows. Suppose there exists $i \neq j$ such that $c_i \neq 0$ and $c_j \neq 0$, and denote by $k_2^\perp$ the dimension of $C_2^\perp$. The code $C_2$ having minimum distance $d_2 > 2$ means that there exists $b_1 = (b_{1,1}, \dots, b_{n_2,1}) \in C_2^\perp$ with $b_{i,1} = 0$ and $b_{j,1} \neq 0$. We complete $(b_1)$ into a basis $(b_k \coloneqq (b_{1,k}, \dots, b_{n_2,k}))_{k \in \llbracket 1, k_2^\perp \rrbracket }$ of $C_2^\perp$ such that
\[
\forall k \in \llbracket 1, k_2^\perp \rrbracket, \, b_{i,k} = 0 \text{ or } b_{j,k} = 0 \, .
\]
Let $k \in \llbracket 1, k_2^\perp \rrbracket$, and consider the matrix multiplication $^tb_k \times c$. Since $b_k \in C_2^\perp$, $^tb_k \times c \in C_1$. However, by our choice of basis $(b_k)_k$, $\wt(c) > \wt(^tb_k \times c)$ and since $\wt(c)$ is the minimum distance of $C_1$, $^t b_k \times c=0$. But $(b_k)_k$ is a basis of $C_2^\perp$, meaning $c \in \F_q^{n_1} \otimes C_2$ whose minimum distance is $d_2$, and $d_2 \leq \wt(c) = d_1 < d_2$, leading to a contradiction.
\end{proof}

\subsubsection{Dual Distance of $C_e(a,b)$}

We are now able to provide the minimum distance of the dual code $C_e(a,b)^\perp$.

\begin{theorem}\label{prop:dualdistance}
Let $(a,b) \in \mathbb{N}^2$ such that $1 \leq a \leq q-1$ and $0 \leq b-ea \leq q-1$. Denote $d_e(a,b)^\perp$ the minimum distance of $C_e(a,b)^\perp$.
Then 
\[
 d_e(a,b)^\perp = \min\{ a,b-ea \} +2  .
\] 
\end{theorem}

\begin{proof}
To get the upper bound, we use the observation that puncturing $C_e(a,b)$ at the points of $\mathcal{H}_e \setminus S_e$ gives the code $\PRS_q(b-ea+1)$ (see Figure \ref{fig:HirzebruchCode}) whose dual distance is $b-ea+2$. Taking a non-zero codeword of minimum weight in $\PRS_q(b-ea+1)^\perp$ and extending its coordinates by $0$ on all points of $\mathcal{H}_e \setminus S_e$ gives a codeword of $C_e(a,b)^\perp$ of weight $b-ea+2$. Using the same technique on the coordinates of $F_e$ gives a codeword of $C_e(a,b)^\perp$ of weight $a+2$. Thus
\[
d_e(a,b)^\perp \leq \min \{ a,b-ea \} + 2 \, . 
\]

For the opposite inequality, recall that, by Proposition \ref{prop: Code Hirzebruch}, we have
\[
C_e(a,b) = \sum_{d_1 = 0}^a \left[ \langle ( 0^{d_1}, \alpha_1^{a-d_1}, \dots, \alpha_q^{a-d_1} ) \rangle \otimes  \PRS_q(b-ea+ed_1+1) \right] \, .
\]
Define $\widetilde{C} \coloneqq \PRS_q(a+1) \otimes (0, \RS_q(b-ea))$. By Lemma \ref{lem:prs}, 
\[
(0, \RS_q(b-ea)) = \bigcap_{d_1=0}^a \PRS_q(b-ea + ed_1 +1) \, ,
\]
and $\widetilde{C} = \sum_{d_1=0}^a \left[  \langle ( 0^{d_1}, \alpha_1^{a-d_1}, \dots, \alpha_q^{a-d_1} ) \rangle \otimes  (0, \RS_q(b-ea))  \right]$ is a subcode of $C_e(a,b)$. Thus
\begin{align*}
C_e(a,b)^\perp &\subset \widetilde{C}^\perp \\
			  &=  \PRS_q(q - a) \boxplus_{\F_q} \left[  \langle (1, 0, \dots, 0) \rangle + (0, \RS_q(q - b + ea)) \right] \\
			  &= \left[ \PRS_q(q-a) \boxplus_{\F_q} (0,\RS_q(q- b + ea)) \right] + E_1 \, ,
\end{align*}
with $E_1 \coloneqq \F_q^{q+1} \otimes  \langle (1, 0, \dots, 0) \rangle$. The space $E_1$ is the set of all square matrices of size $(q+1)^2$ whose first row is an element of $\F_q^{q+1}$ and the $q$ other rows are zero. We write $\widetilde{C}^\perp$ as the direct sum of two vector spaces whose supports are disjoints:
\[
\widetilde{C}^\perp = E_1 \oplus E_2 \, ,
\] 
where
\[
E_2 \coloneq \left[ \PRS_q(q-a) \otimes (0,\F_q^q) \right] + \left[ \F_q^{q+1} \otimes (0, \RS_q(q-b+ea)) \right] \, .
\]
Note that the minimum distances of $E_1$ and $E_2$ are, respectively, $d(E_1) = 1$ and $d(E_2) = \min \{ a+2 ,b-ea+1  \}$ by Equality \eqref{eq:distdualtensor} of Proposition \ref{prop:checkproduct}, meaning $d(\widetilde{C}^\perp ) = 1$. Let $c \in C_e(a,b)^\perp \subset \widetilde{C}^\perp$ and suppose $\wt( c ) < \min \{a+2, b-ea+1 \}$. Write $c = c_1 + c_2$ where $c_1 \in E_1$ and $c_2 \in E_2$. The words $c_1$ and $c_2$ have disjoint supports and $\wt(c) = \wt(c_1) + \wt(c_2)$, but since $d(E_2) =  \min \{ a+2 ,b-ea+1  \} > \wt(c)$ then $c_2 = 0$ and $c \in E_1$. This means that $c$ is a matrix of size $(q+1)^2$ with every rows but the first one equal to zero. However, the puncturing  of $C_e(a,b)$ at the points of $\mathcal{H}_e \setminus \{ F_e \}$ (\ie where we only keep the first row) is the code $\PRS_q(a+1)$ whose dual distance is $a+2 > \wt (c)$, leading to a contradiction. This entails
\[
\min\{ a+2, b-ea +1 \} \leq d_e(a,b)^\perp \, .
\]

It remains to improve our lower bound by $1$ when $a+2 \geq b-ea+2$. Suppose $b-ea \leq a$, and let $c' \in E_1 \oplus E_2 $ be a codeword of weight $b-ea+1$. As we saw, by the disjoints supports of $E_1$ and $E_2$ and since $b-ea+1 < a+2$, $c'$ is an element of $E_2$. By the inclusion
\begin{align*}
E_2 &= \left[  \PRS_q(q-a) \otimes (0,\F_q^q) \right] + \left[  \F_q^{q+1} \otimes (0, \RS_q(q-b+ea)) \right] \\
      &\subset \left[ \PRS_q(q-a) \otimes \F_q^{q+1} \right] + \left[ \F_q^{q+1} \otimes \PRS_q(q-b+ea+1) \right] \\
      &= \PRS_q(q-a) \boxplus_{\F_q} \PRS_q(q-b+ea+1) \\ &= \left[ \PRS_q(a+1) \otimes \PRS_q(b-ea) \right]^\perp \, ,
\end{align*}
the word $c'$ is a word of $\left[ \PRS_q(a+1) \otimes \PRS_q(b-ea) \right]^\perp$ of minimum weight. By Lemma \ref{lem:lemmetechnique}, the support of $c'$ corresponds to some $\F_q$-rational points of the affine line $\Sigma_{\text{aff}} \coloneqq \Sigma_e \setminus \{ \Sigma_e \cap F_e \}$, where $\Sigma_e$ is either $S_e$ or $\sigma_e$ (see Figure \ref{fig:HirzebruchDiviseurs}). However, puncturing $C_e(a,b)$ at the points of $\mathcal{H}_e \setminus \Sigma_{\text{aff}}$ gives the code $\RS_q(b+1)$ or $\RS_q(b-ea+1)$ whose dual distance is $b+2$ or $b-ea+2$ respectively, which is strictly more than $b-ea +1$, meaning $c' \notin C_e(a,b)^\perp$ and improving our inequality
\[
d_e(a,b)^\perp \geq \min \{ a, b-ea \} +2 \, .
\]

\end{proof}

\section{Puncturing $C_e(a,b)$.}\label{section:3}

We describe in this section a puncturing of $C_e(a,b)$ corresponding to its open $\mathbb{A}^2$. Instead of evaluating the polynomials $P \in \Lis (aS_e+bF_e)$ at all the points of $\mathcal{H}_e (\F_q) = \mathbb{P}^1 (\F_q) \times \mathbb{P}^1(\F_q)$ (Figure \ref{fig:HirzebruchCode}), we will study the images of the evaluation maps on the points of $\mathbb{A}^2(\F_q) = \F_q \times \F_q$, these points being the $\F_q$-rational points of $\mathcal{H}_e \setminus \{ S_e \cup F_e \}$(see Figure \ref{fig:HirzebruchDiviseurs}). We will compute its parameters, including the dual distance, before giving an explicit form of the dual. Expliciting the dual of length $q^2$ will give the set of all the codewords $c \in C_e(a,b)^\perp$ that are zero on the first row and first column. Computing the dual $C_e(a,b)^\perp$ in Section \ref{section:4} will only be a matter of finding every codeword that is non-zero on its first row or first column.

\begin{remark}
Originally, AG codes on the Hirzebruch surfaces were studied on the $(q-1)^2$ rational points of $\{X_1X_2T_1T_2 \neq 0 \}$ by Hansen \cite{hansen2002toric} in the injective case $b \leq q-2$. They correspond to the puncturing of $C_e(a,b)$ at the first two rows and columns. The parameters and properties for these codes were studied in the original paper \cite{hansen2002toric}, and later in \cite{little2006toric} and \cite{hansen2012quantum}. These are also particular cases of the \emph{generalized toric codes} described in \cite{ruano2009structure}, along with their duals. One may adapt the proofs of this section to recover the parameters, dual and dual distance of the original toric code of length $(q-1)^2$.
\end{remark}

\subsection{Puncturing at $\{ X_1T_1=0 \}$ }
We consider the puncturing of $C_e(a,b)$ at the points of $S_e \cup F_e$, and we define $U_{e,q} \coloneqq \left( \mathcal{H}_e \setminus \{ S_e \cup F_e \}\right) (\F_q)$. In our representation, this set is given by
\[
U_{e,q} = \{ (1,\alpha,1, \beta) \, | \, (\alpha,\beta) \in \F^2_q \} \simeq \mathbb{A}^2(\F_q) \, .
\]

\subsubsection{The evaluation code $C_{\mathbb{A},e}(a,b)$}

\begin{definition}[$C_{\mathbb{A},e}(a,b)$]
Let $e \in \mathbb{N}$ and $(a,b) \in \mathbb{N}^2$. We denote by $C_{\mathbb{A},e}(a,b)$ the puncturing of $C_e(a,b)$ at the $2q+1$ coordinates of the first row and first column.
\end{definition}

\begin{proposition}\label{prop:codeq}
Let $a,b \in \mathbb{N}$, and let $\mathbb{F}_q = \{ \alpha_1, \dots, \alpha_q \}$. Define the evaluation morphism at the points of $U_{e,q}$ by
\[
\ev_{U_{e,q}}: \left\{
\begin{array}{r c l}
\Lis(aS_e + bF_e) & \rightarrow & \F_q^{q^2} \\
f & \mapsto & (f(P))_{P \in U_{e,q}}
\end{array}
\right . \, . 
\]
Then, $C_{\mathbb{A},e}(a,b) = \Image(\ev_{U_{e,q}})$, and an explicit form for this code is given by
\[
C_{\mathbb{A},e}(a,b) = \sum_{d = 0}^a \left[ \langle (\alpha_1^d, \dots, \alpha_q^d) \rangle \otimes \RS_q(b - ed +1) \right] \, .
\]
\end{proposition}

\begin{proof}
By our choices of representation, puncturing at the first row and the first column is the same as evaluating the polynomials of $\Lis(aS_e + bF_e)$ at the $\F_q$-rational points of $\mathcal{H}_e \setminus \{ S_e \cap F_e \}$ (see Figure \ref{fig:HirzebruchCode}), thus $C_{\mathbb{A},e}(a,b) = \Image(\ev_{U_{e,q}})$. Furthermore, by puncturing the codes at the first coordinates on both sides of the tensor product of \ref{prop: Code Hirzebruch}, we get the explicit form of $C_{\mathbb{A},e}(a,b)$.  
\end{proof}
\begin{notation}
As for $C_e(a,b)$, we will also denote $C_{\mathbb{A},e}(a\sigma_e + (b-ea)F_e) \coloneqq C_{\mathbb{A},e}(a,b)$.
\end{notation}

Since we evaluate the polynomials $P \in \Lis(aS_e + bF_e)$ at points of $\{X_1T_1 \neq 0 \}$, a generating family of $C_{\mathbb{A},e}(a,b)$ is given by the codewords of the form $\ev_{U_{e,q}}(M)$, where $M = X^dT^c \in \F_q[X,T]$ and
\[
 \left\{
\begin{array}{c}
b-ed \geq c,\\
a \geq d
\end{array}.
\right . \,
\]
As such, the codes $C_{\mathbb{A},e}(a,b)$ are subcodes of Reed--Muller codes over $\mathbb{A}^2$, whose definition and minimum distance are recalled below.

\begin{definition}[Reed--Muller codes on the affine plane.]
Let $k \in \mathbb{N}$. The \emph{affine Reed--Muller code on the plane}, denoted $\RM(k,2,q)$, is the image of the evaluation morphism

\[
\ev_{\mathbb{A}^2} : \left\{ \begin{array}{r c l} 
\F_q[X,T]_{k}^h &\rightarrow & \F_q^{q^2} \\
P &\mapsto &(P(x,y))_{(x,y) \in \mathbb{A}^2} 
\end{array}\right. \, .
\]
Let $\F_q = \{ \alpha_1, \dots, \alpha_q \}$. An explicit form of $\RM(k,2,q)$ is
\[
\RM(k,2,q) = \sum_{d=0}^k \left[ \langle (\alpha_1^d, \dots, \alpha^d_q) \rangle \otimes \RS_q(k - d + 1) \right] = 	C_{\mathbb{A},1}(k,k) \, .
\]
\end{definition}

\begin{proposition}[{\cite[2.6.2]{paterson2000generalized}}]
Let $k \in \mathbb{N}$. Write $k = r(q-1) + s$ with $r \in \mathbb{N}$ and $0 \leq s < q-1$. The minimum distance $d(\RM(k,2,q))$ is given as follow:
\begin{itemize}
\item if $r =0$, then
\[
d(\RM(k,2,q)) = q(q-s) \, ;
\]
\item if $r=1$, then
\[
d(\RM(k,2,q)) = q-s \, ;
\]
\item and if $r \geq 2$, then
$d(\RM(k,2,q)) = 1$.
\end{itemize}
\end{proposition}

\begin{remark}
Suppose $b-ea \leq 0$ and $e \geq 1$. The space $\Lis(aS_e + bF_e)$ is isomorphic to the set of all weighted polynomials in $\F_q[T,X]$ of weighted degree less than $b$ with weight distribution $(1,e)$. In that case, $C_{\mathbb{A},e}(a,b)$ is the \emph{weighted affine Reed--Muller code} $\WARM (b,2,q)$, well-studied in \cite{sorensen1992weighted}. In particular, by \cite[Theorem 3]{sorensen1992weighted}, $\WARM(b,2,q)^\perp = \WARM((q-1)(1+e) - b - 1,2,q)$, and consequently
\[
C_{\mathbb{A},e}(a,b)^\perp = C_{\mathbb{A},e}(\widetilde{a}, \widetilde{b}) \, ,
\]
where $\widetilde{b} = (q-1)(1+e) - b - 1$ and $\widetilde{b} - e \widetilde{a} \leq 0$.
\end{remark}

Removing the points of $S_e$ and $F_e$ makes it so the codes $C_{\mathbb{A},e}(a,b)$ are parametrized by divisors that avoid the points of evaluation, allowing inclusions between some of the codes.

\begin{proposition}
For all $(a_1,b_1), (a_2,b_2) \in \mathbb{N}^2$,		
\[
 (a_1 \leq a_2 \text{ and } b_1 \leq b_2) \Rightarrow C_{\mathbb{A},e}(a_1, b_1) \subset C_{\mathbb{A},e}(a_2,b_2) \, .
\]
Furthermore, for all integer $e' \geq e$ and all $(a,b) \in \mathbb{N}^2$,
\[
C_{\mathbb{A},e'} (a,b) \subset C_{\mathbb{A},e} (a,b) \, .
\]
\end{proposition}

\begin{proof}
The statement follows by using Proposition \ref{prop:codeq} together with the fact that when $k \leq k' $ we have $\RS_q(k) \subset \RS_q(k')$.
\end{proof}

\subsubsection{Parameters of $C_{\mathbb{A},e}(a,b)$}
We will now compute the parameters of the $[q^2, k_{\mathbb{A},e}(a,b), d_{\mathbb{A},e}(a,b)]_q$-code $C_{\mathbb{A},e}(a,b)$.

\begin{proposition}[Parameter of $C_{\mathbb{A},e}(a,b)$]\label{prop:parametreq}
We fix $e \geq 1$. Let $(a,b) \in \mathbb{N}^2$ such that $0\leq a \leq q-1$ and $0 \leq b-ea$. Let us denote $[q^2, k_{\mathbb{A},e}(a,b),d_{\mathbb{A},e}(a,b)]_q$ the parameters of $C_{\mathbb{A},e}(a,b)$.

We set 
\[ s_{\mathbb{A}} \coloneqq \frac{b+1-q}{e} \, \text{ and } \widetilde{s}_{\mathbb{A}} \coloneqq \left\{ \begin{array}{c} \min\{ \lfloor s_{\mathbb{A}} \rfloor, a \} \text{ if } s_{\mathbb{A}} \geq 0 \\ -1 \text{ otherwise } \end{array} \right. \, .\]
Then, the dimension of $C_{\mathbb{A},e}(a,b)$ is
\[ k_{\mathbb{A},e}(a,b) = (\widetilde{s}_{\mathbb{A}} + 1)q + (a-\widetilde{s}_{\mathbb{A}}) \left( b+1 - e \left( \frac{a + \widetilde{s}_{\mathbb{A}} +1}{2} \right) \right) \, .
\]
Its minimum distance $d_{\mathbb{A},e}(a,b)$ is given as follow:
\begin{itemize}
\item if $q-1 > b$, then
\[ d_{\mathbb{A},e} = q(q-b) \, ; \]
\item if $b-ea < q-1 \leq b$, then
\[ d_{\mathbb{A},e}(a,b) = q - \left\lfloor \frac{b-q+1}{e} \right\rfloor \, ; \]
\item and if $q-1 \leq b-ea$, then
\[ d_{\mathbb{A},e}(a,b) = q-a \, .\]
\end{itemize}
\end{proposition}

\begin{proof}
Denote $\F_q = \{\alpha_1, \dots, \alpha_q \}$ and for all $d \in \mathbb{N}$ write $\alpha^{*d}= (\alpha_1^d, \dots, \alpha_q^d)$. Since $a \leq q-1$, the family of vectors $( \alpha^{*d} )_{d \in \llbracket 0, a \rrbracket}$ is a basis of $\RS_q(a+1)$ and, thus, linearly independent. For all $d \in \llbracket 1,a \rrbracket$, we have $b-ed+1 \geq q$ if and only if $d \leq \lfloor (b+1-q)/e \rfloor$. Using Proposition \ref{prop:parametreq}, $C_{\mathbb{A},e}(a,b)$ can be expressed as the following direct sum

\[
C_{\mathbb{A},e}(a,b) = (\RS_q(\widetilde{s}_{\mathbb{A}}+1) \otimes \F_q^q) \oplus \bigoplus_{d=\widetilde{s}_{\mathbb{A}}+1}^a  \left[ \langle (\alpha^{*d}) \rangle \otimes \RS_q(b - ed +1) \right] \, .
\]

Whence its dimension $k_{\mathbb{A},e}(a,b)$ is easily computed as

\begin{align*}
k_{\mathbb{A},e}(a,b) &= (\widetilde{s}_{\mathbb{A}} +1)q + \left( \sum_{d= \widetilde{s}_{\mathbb{A}}+1}^a b-ed+1 \right) \\
&= (\widetilde{s}_{\mathbb{A}} + 1)q + (a-\widetilde{s}_{\mathbb{A}}) \left( b+1 - e \left( \frac{a + \widetilde{s}_{\mathbb{A}} +1}{2} \right) \right) \, .
\end{align*}

We will now show that $C_{\mathbb{A},e}(a,b)$ is included in an affine Reed-Muller code over $\mathbb{A}^2$, whose minimum distance will give a lower bound for $d_{\mathbb{A},e}(a,b)$. We distinguish two cases (which correspond to the injective and the non-injective cases, respectively):
\begin{enumerate}

\item Suppose $b < q-1$, then $\widetilde{s}_{\mathbb{A}}=-1$, and since $e \geq 1$, we have 
\begin{align*}
C_{\mathbb{A},e}(a,b) &\subset \bigoplus_{d=0}^{q-1} \left[ \langle (\alpha^{*d}) \rangle \otimes \RS_q(b - d +1) \right] \\
&= \RM(b,2,q) \, .
\end{align*}
By this inclusion, $d_{\mathbb{A},e}(a,b) \geq d(\RM(b,2,q)) = q(q-b)$, and since $\widetilde{C} \coloneqq \langle ( 1, \dots, 1) \otimes \RS_q(b+1) \subset C_{\mathbb{A},e}(a,b)$ with $d(\widetilde{C})=q \times d(\RS_q(b+1)) = q(q-b)$, we obtain the equality \[d_{\mathbb{A},e}(a,b) = q(q-b) \, . \]

\item Suppose $b \geq q-1$, meaning $\widetilde{s}_{\mathbb{A}}= \min\{ \lfloor \frac{b+1-q}{e} \rfloor , a \} \geq 0$. If $\widetilde{s}_{\mathbb{A}} = a$ then $C_{\mathbb{A},e}(a,b) = \RS_q(a+1) \otimes \F_q^q$ and $d_{\mathbb{A},e}(a,b) = q-a$. Suppose now that $\widetilde{s}_{\mathbb{A}} = \lfloor \frac{b+1-q}{e} \rfloor$. By construction of $\widetilde{s}_{\mathbb{A}}$, for all $d \in \llbracket 0,\widetilde{s}_{\mathbb{A}} \rrbracket$ we have $\RS_q(b-ed+1) = \F_q^q$, and $\RS_q((b+1)-e(\widetilde{s}_{\mathbb{A}}+1)) \subset \RS_q(q-1) = \RS_q((q +\widetilde{s}_{\mathbb{A}}) - (\widetilde{s}_{\mathbb{A}}+1))$. Since $e \geq 1$, we get
\begin{align*}
C_{\mathbb{A},e}(a,b) &\subset \bigoplus_{d=0}^{q-1} \left[ \langle (\alpha^{*d}) \rangle \otimes \RS_q((q + \widetilde{s}_{\mathbb{A}})  - d) \right] \\
&= \RM(q+\widetilde{s}_{\mathbb{A}}-1,2,q) \, .
\end{align*}
Since $0 \leq \widetilde{s}_{\mathbb{A}} \leq a \leq q-1$, one deduces that
\[
d_{\mathbb{A},e}(a,b) \geq d(\RM_q(q-1+ \widetilde{s}_{\mathbb{A}})) \geq q - \widetilde{s}_{\mathbb{A}} \, .
\]
Finally, as $\RS_q(\widetilde{s}_{\mathbb{A}}+1) \otimes \F_q^q \subset C_{\mathbb{A},e}(a,b)$, we conclude
\[
d_{\mathbb{A},e}(a,b) = q-\widetilde{s}_{\mathbb{A}} \, .
\]
\end{enumerate}
\end{proof}

 We end this subsection by giving the size of the kernels for the evaluation and the puncturing morphisms, from which we deduce the difference of dimensions between the duals of $C_e(a,b)$ and $C_{\mathbb{A},e}(a,b)$. Recall that for a fixed $(a,b) \in \mathbb{N}^2$ we defined two integers $\widetilde{s}$ and $\widetilde{s}_\mathbb{A}$ in Propositions \ref{prop:paramcode} and \ref{prop:parametreq}, respectively:

\[
s \coloneqq \frac{b-q}{e} \, , \text{ and } \widetilde{s} \coloneqq \left\{ 
\begin{array}{l} 
\min \{ \lfloor s \rfloor , a \} \text{ if } s \geq 0 \\
-1 \text{ otherwise}
\end{array}
\right. \, ;
\]
\[ s_{\mathbb{A}} \coloneqq \frac{b+1-q}{e} \, , \text{ and } \widetilde{s}_{\mathbb{A}} \coloneqq \left\{ \begin{array}{c} \min\{ \lfloor s_{\mathbb{A}} \rfloor, a \} \text{ if } s_{\mathbb{A}} \geq 0 \\ -1 \text{ otherwise } \end{array} \right. \, . \]
They both allow us to measure how the evaluation morphisms from $\Lis(aS_e+bF_e)$ fail to be injective, with $\widetilde{s}$ being strictly negative if and only if $\ev_{\mathcal{H}_e}$ is injective, and the same for $\widetilde{s}_\mathbb{A}$ and $\ev_{U_{e,q}}$. 
\begin{proposition}\label{coroq}
 With the same notations as in Proposition \ref{prop:parametreq}, the kernel of the evaluation morphism $\ev_{U_{e,q}}: \Lis(aS_e+bF_e) \twoheadrightarrow C_{\mathbb{A},e}(a,b)$ has dimension
\[
\dim \ker(\ev_{U_{e,q}}) = (\widetilde{s}_{\mathbb{A}}+1)\left(b-q +1-\frac{e\widetilde{s}_{\mathbb{A}}}{2} \right) \, .
\]
Consider the puncturing morphism at the coordinates of $S_e \cup F_e$
\[
\pi_{U_{e,q}}: C_e(a,b)  \twoheadrightarrow  C_{\mathbb{A},e}(a,b) \, .
\]
Then, $\dim \ker(\pi_{U_{e,q}}) = \widetilde{s} + 1$. 
\end{proposition}

\begin{proof}
Using the same method as in the proof of Corollary \ref{coro}, the difference of dimensions between $\Lis(aS_e+bF_e)$ and $C_{\mathbb{A},e}(a,b)$ is the sum
\[
\dim \ker(\ev_{U_{e,q}}) = \sum_{d=0}^{\widetilde{s}_\mathbb{A}} \left( b - ed + 1 - q \right)
  = (\widetilde{s}_{\mathbb{A}} + 1) \left( b - q + 1 - \frac{e \widetilde{s}_{\mathbb{A}} }{2} \right) \, .
\]
Consider now the surjective evaluation morphism $\ev_{\mathcal{H}_e}: \Lis(aS_e+bF_e) \twoheadrightarrow C_e(a,b)$ whose dimension for the kernel is computed in Corollary \ref{coro} as $\dim \ker(ev_{\mathcal{H}_e}) = (\widetilde{s}+1)\left(b-q-\frac{e\widetilde{s}}{2} \right)$. By the following commutative diagram
\begin{center}
\begin{tikzcd}
& C_e(a,b) \arrow[dd, "\pi_{U_{e,q}}",twoheadrightarrow] \\
\Lis(aS_e + bF_e) \arrow[ur, "ev_{\mathcal{H}_e}",twoheadrightarrow] \arrow[dr, "ev_{U_{e,q}}",twoheadrightarrow]
	 & \\
& C_{\mathbb{A},e}(a,b) 
\end{tikzcd}
\end{center}
we see that $\dim \ker(\pi_{U_{e,q}}) = \dim \ker( \ev_{U_{e,q}}) - \dim \ker( \ev_{\mathcal{H}_e})$. We now distinguish two cases.
\begin{enumerate}
\item Suppose $\widetilde{s}_\mathbb{A} = \widetilde{s}$, then $\dim \ker( \ev_{U_{e,q}}) = \dim \ker( \ev_{\mathcal{H}_e}) + (\widetilde{s} +1)$, and thus \[ \dim \ker(\pi_{U_{e,q}}) = \widetilde{s} + 1 \, . \]
\item Suppose $\widetilde{s}_\mathbb{A} \neq \widetilde{s}$. It is the case if and only if $\widetilde{s}_\mathbb{A} = \widetilde{s} + 1$, meaning $\widetilde{s}_\mathbb{A} = \frac{b-q+1}{e} \in \mathbb{N}$ and $e\widetilde{s} = b-q+1-e$. Injecting these values we get
\begin{align*}
\dim \ker(\ev_{U_{e,q}}) &= (\widetilde{s}_\mathbb{A} +1) \left( \frac{b-q+1}{2} \right) \, ; \\
\dim \ker(\ev_{\mathcal{H}_e}) &= \widetilde{s}_\mathbb{A} \left(\frac{b-q-1+e}{2}   \right) \, .
\end{align*}
We can now easily conclude
\begin{align*}
\dim \ker (\pi_{U_{e,q}}) &= \widetilde{s}_\mathbb{A} \left( 1 - \frac{e}{2} \right) + \frac{b-q+1}{2}\\
&= \widetilde{s}_\mathbb{A} + \frac{b-q+1-e\widetilde{s}_\mathbb{A}}{2} = \widetilde{s} +1 \, .
\end{align*}
\end{enumerate}
\end{proof}

\begin{corollary}\label{diffdimdual}
With the same notations as in Proposition \ref{coroq}, we have
\[
\dim(C_e(a,b)^\perp) = \dim(C_{\mathbb{A},e} (a,b)^\perp) + 2q - \widetilde{s} \, .
\]
\end{corollary}

\begin{proof}
We have $\dim(C_e(a,b)^\perp) = (q+1)^2 - \dim(C_e(a,b))$ and $\dim(C_{\mathbb{A},e}(a,b)^\perp) = q^2 - \dim(C_{\mathbb{A},e}(a,b))$. From Proposition \ref{coroq}, we have $\dim(C_e(a,b)) = \dim(C_{\mathbb{A},e}(a,b)) + \widetilde{s} + 1$, which allows us to conclude.
\end{proof}

\subsection{Dual and dual distance of $C_{\mathbb{A},e}(a,b)$}

Removing the points at infinity and allowing the codes to be included in each other makes it easier to study the dual of $C_{\mathbb{A},e}(a,b)$. For example, we can compute its minimum distance and even an explicit form, as shown in the two following theorems.

\begin{theorem}\label{prop:dualdistanceq}
Let $(a,b) \in \mathbb{N}^2$ such that $1 \leq a \leq q-2$ and $0 \leq b-ea \leq q-2$. Denote $d_{\mathbb{A},e}(a,b)^\perp$ the minimum distance of $C_{\mathbb{A},e}(a,b)^\perp$.
Then 
\[
\min\{ a,b-ea \} +2 \leq d_{\mathbb{A},e}(a,b)^\perp \leq \min\{ a, b \} +2 \, .
\] 
\end{theorem}

\begin{proof}
From the explicit form given in Proposition \ref{prop:codeq}, since for all $d \in \llbracket 0,a \rrbracket$ we have $\RS_q(b-ea+1) \subset \RS_q(b-ed+1) \subset \RS_q(b+1)$, the space $C_{\mathbb{A},e}(a,b)$ is framed as follow
\[
\left[ \RS_q(a+1) \otimes \RS_q(b-ea+1) \right]  \subset C_{\mathbb{A},e}(a,b) \subset \left[ \RS_q(a+1) \otimes \RS_q(b+1) \right] \, .
\]
By taking the duals we obtain the inclusions
\begin{align*}
\RS_q(q-a-1) \boxplus_{\F_q} \RS_q(q-b+ea-1) &\supset C_{\mathbb{A},e}(a,b)^\perp \\
&\supset \RS_q(q-a-1) \boxplus_{\F_q} \RS_q(q-b-1) \, .
\end{align*}
Finally, we conclude by Proposition \ref{prop:checkproduct} that
\[
\min\{ a,b-ea \} +2 \leq d_{\mathbb{A},e}(a,b)^\perp \leq \min\{ a, b \} +2 \, .
\] 
\end{proof}

\begin{theorem}\label{theo:dualq}
Let $a,b \in \mathbb{N}$, with $b-ea \geq 0$. Let us denote $C_1 \coloneqq C_{\mathbb{A},e}((q-1)\sigma_e + (q-2-b)F_e)$ and $C_2 \coloneqq C_{\mathbb{A},e}((q-2-a)\sigma_e + (q-1)F_e)$. Then
\[
C_{\mathbb{A},e}(a,b)^\perp = C_1 + C_2 \, .
\]
\end{theorem}

\begin{proof}
We first show the right to left inclusion. By Proposition \ref{prop:codeq}, since $(q-2-a)\sigma_e + (q-1)F_e \simeq (q-2-a)S_e + ((q-1)+e(q-1-a))F_e$, the code $C_2$ has a simple explicit form
\begin{align*}
C_2 = \sum_{d=0}^{q-2-a}  \left[ \langle (\alpha_1^d, \dots, \alpha_q^d) \rangle \otimes \F_q^q \right] &= \RS_q(q-a-1) \otimes \F_q^q \\
&\subset  \RS_q(q-a-1) \boxplus_{\F_q} \RS_q(q-b-1) \, .  
\end{align*}
We saw in the proof of Theorem \ref{prop:dualdistanceq} that $\RS_q(q-a-1) \boxplus_{\F_q} \RS_q(q-b-1) \subset C_{\mathbb{A},e}(a,b)^\perp$, and thus $C_2 \subset C_{\mathbb{A},e}(a,b)^\perp$.

Take $M'=X^{d'}T^{c'} \in \Lrr((q-1)\sigma_e + (q-2-b)F_e)(\{X_1T_1 \neq 0 \})$ and $M=X^dT^c \in \Lrr(aS_e + bF_e)(\{X_1T_1 \neq 0 \})$, \ie with
\[
 \left\{
\begin{array}{c}
(q-b-2)+e(q-1-d') \geq c',\\
q-1 \geq d'
\end{array}
\right. \text{ and }
 \left\{
\begin{array}{c}
b-ed \geq c,\\
a \geq d
\end{array}.
\right.  
\]
Let us denote $w' \coloneqq \ev_{U_{e,q}}(M')$ and $w \coloneqq \ev_{U_{e,q}}(M)$. Since $w'$ and $w$ run through generating families of  $C_1$ and $C_{\mathbb{A},e}(a,b)$, respectively, we want to show that $ w \cdot w'  = 0$. Two cases arise:
\begin{enumerate}
\item Suppose $d' \leq q-2-d$. Then, looking at the rows, we see that $w' \in \RS_q(d'+1) \otimes \F_q^q \subset \RS_q(q-2-d) \otimes \F_q^q$ and $w \in \RS_q(d+1) \otimes \F_q^q = \left( \RS_q(q-2-d) \otimes \F_q^q \right)^\perp$, so $ w \cdot w'  = 0$.
\item Suppose $d' \geq q-1-d$. Then, $c' \leq (q-b-2)+e(q-1-d') \leq q-2-c$, and ,by the same argument as in the previous case, by looking at the columns, one gets $ w \cdot w'  = 0$. 
\end{enumerate}
We just showed
\[
C_1 + C_2 \subset C_{\mathbb{A},e}(a,b)^\perp \, .
\]

We now show the equality by the dimensions. Using the explicit forms of $C_1$ and $C_2$ by Proposition \ref{prop:codeq}, and since $((\alpha_1^d, \dots, \alpha_q^d))_{d \in \llbracket 0,q-1 \rrbracket}$ forms a basis of $\F_q^q$, $C_1+C_2$ can we written in a direct sum
\[
C_1 + C_2 = \bigoplus_{d=0}^{q-1} \left[ \langle (\alpha_1^d, \dots, \alpha_q^d) \rangle \otimes F_d \right],
\]
where
\[
\forall d \in \llbracket 0,q-1 \rrbracket, \, F_d = \left\{
\begin{array}{c r}
\F_q^q &\text{ if } d \leq q-2-a,\\
\RS_q(q-1-b+e(q-1-d)) &\text{ if } d > q-2-a
\end{array}
\right. \, .
\]
Recall that for all $k \leq 0$, $\RS_q(k) = \{ 0 \}$. Using Proposition \ref{prop:parametreq}, some straightforward calculations show that
\[
\dim (C_1 + C_2) = \sum_{d=0}^{q-1} \dim (F_q) = q^2 - \dim C_{\mathbb{A},e}(a,b) = \dim C_{\mathbb{A},e}(a,b)^\perp \, .
\]
\end{proof}

\section{An explicit form of the dual}\label{section:4}
This section is solely devoted to giving an explicit form of the dual $C_e(a,b)^\perp$. Recall that in Theorem \ref{theo:dualq}, we were able to give the dual $C_{\mathbb{A},e}(a,b)^\perp$ of the puncturing of $C_e(a,b)$ at the points of $S_e \cup F_e$, \ie on its first row and first column. Extending every codeword $m \in C_{\mathbb{A},e}(a,b)^\perp$ by zero coordinates on the first row and first column gives the exact set of \emph{every} codeword $\overline{m} \in C_e(a,b)^\perp$ that have zero coordinates on the first row and first column. Denote this set by $\overline{C_{\mathbb{A},e}(a,b)^\perp}$, then 
\[
\overline{C_{\mathbb{A},e}(a,b)^\perp} \subset C_e(a,b)^\perp \, ,  \] and by Corollary \ref{diffdimdual}, \[ \dim(C_e(a,b)^\perp) - \dim \left( \overline{C_{\mathbb{A},e}(a,b)^\perp} \right) = 2q - \widetilde{s} \leq 2q +1 \, .
\]
We will give an explicit form of $C_e(a,b)^\perp$ by completing any basis of $\overline{C_{\mathbb{A},e}(a,b)^\perp}$ with $2q - \widetilde{s}$ suitable elements. To this end, we begin by formalizing the extension of a code $C_{\mathbb{A},e}(a,b)$ to the length $(q+1)^2$ by adding zero coordinates.

\begin{definition}
Let $(a,b) \in \mathbb{N}^2$. We define the code $\overline{C_{\mathbb{A},e}}(a,b)$ of length $(q+1)^2$ to be the set of all codewords of $C_{\mathbb{A},e}(a,b)$ extended by zero coordinates on the first row and first column. 

We will also denote $\overline{C_{\mathbb{A},e}}(a\sigma_e + bF_e) \coloneqq \overline{C_{\mathbb{A},e}}(a,b-ea)$.
\end{definition}

\begin{theorem}\label{theo:dual}
Let $a,b \in \mathbb{N}$ such that $a \leq q-2$ and $0 \leq b-ea$. Let $\F_q = (\alpha_1, \dots, \alpha_q)$ and, for all $k \in \mathbb{N}$, denote $\alpha^{*k} \coloneqq (\alpha_1^k, \dots, \alpha_q^k)$. We set
\[ s \coloneqq \frac{b-q}{e} \, \text{ and } \, \widetilde{s} \coloneqq \left\{ \begin{array}{c} \min\{ \lfloor s \rfloor, a \} \text{ if } s \geq 0, \\ -1 \text{ otherwise } \end{array} \right. \, . \]
Define for all $d \in \llbracket 1 + \widetilde{s},a \rrbracket$ the vector $v_d \coloneqq (0, \alpha^{*(q-1-d)}) \otimes (1, \alpha^{*(q-b+ed)}) \in \F_q^{(q+1)^2}$.
Define three codes of length $(q+1)^2$ by:
\begin{align*}
C_1 &\coloneqq \overline{C_{\mathbb{A},e}}((q-1)\sigma_e + (q-2-b)F_e) \, , \\
C_2 &\coloneqq \PRS_q(q-a) \otimes \F_q^{q+1} \, , \\
C_3 &\coloneqq \langle (v_d)_{d \in \llbracket 1 + \widetilde{s},a \rrbracket} \rangle \, . 
\end{align*}
Then the dual of $C_e(a,b)$ has the explicit form
\[
C_e(a,b)^\perp = C_1 + C_2 + C_3 \, .
\]
\end{theorem}

\begin{proof}
We start by showing the right to left inclusion, while computing the dimension of $C_1 + C_2 + C_3$.
As we saw from Theorem \ref{theo:dualq}, if we denote $E_1 \coloneqq C_1 + \overline{C_{\mathbb{A},e}}((q-2-a)\sigma_e + (q-1)F_e)$ then we have
\[  E_1 \subset C_e(a,b)^\perp \, . \]
The left-hand side has dimension $\dim(E_1) = k^\perp_{e,q}(a,b)$, the dimension of $C_{\mathbb{A},e}(a,b)^\perp$.
However, we have $C_{\mathbb{A},e}((q-2-a)\sigma_e + (q-1)F_e) = \RS_q(q-1-a) \otimes \F_q^q$, meaning that we can rewrite
\[ E_1 = C_1 + \left( (0,\RS_q(q-1-a)) \otimes (0,\F_q^q) \right) \, .\]
Now, it holds that $(0,\RS_q(q-1-a)) \subset \PRS_q(q-a)$ and $(0,\F_q^q) \subset \F_q^{q+1}$, and since $C_e(a,b) \subset \PRS_q(a+1) \otimes \F_q^{q+1}$, we entail
\[
E_2 \coloneqq C_1 + C_2 \subset C_e(a,b)^\perp \, ,
\]
and
\[
E_2 = E_1 + \left[ \PRS_q(q-a) \otimes (1,0,\dots,0) \right] + \left[ (1, \alpha^{*q-a-1})\otimes (0,\F_q) \right] \, .
\]
The spaces $\PRS_q(q-a) \otimes (1,0,\dots,0)$ and $(1, \alpha^{*q-a-1})\otimes (0,\F_q)$, of dimensions $q-a$ and $q$ respectively, have disjoint supports and thus are in direct sum. Since every coordinate on the first row and the first column of a codeword of $ E_1$ is zero, we get $\dim(E_2) = \dim(E_1) + 2q-a$. Finally, let $M=X_1^{d_1}X_2^{d_2}T_1^{c_1}T_2^{c_2} \in L(aS_e + bF_e)$ and fix $d \in \llbracket 1+ \widetilde{s},a \rrbracket$. To show that $C_3$ is included in $C_e(a,b)^\perp$, we only need to show that $\ev_{\mathcal{H}_e}(M) \cdot v_d = 0$. Observe that the scalar product of $\ev_{\mathcal{H}_e}$ and $v_d$ is the product on $\F_q$ of two scalar products (see Figure \ref{fig:HirzebruchCode})
\[
\ev_{\mathcal{H}_e}(M) \cdot v_d =  \left[ ( 0^{d_1}, \alpha^{*d_2} ) \cdot  (0, \alpha^{*(q-1-d)}) \right] \times \left[  ( 0^{c_1}, \alpha^{*c_2} ) \cdot (1, \alpha^{*(q-b+ed)}) \right] \, .
\]
We distinguish two cases:
\begin{itemize}
\item Suppose $d \neq d_2$. Since both $d$ and $d_2$ are smaller than $ a \leq q-2$, we either have $q-1$ does not divide $q-1 -d + d_2$ or $q-1 - d -d_2 = 0$. In both cases, we get
\[
  ( 0^{d_1}, \alpha^{*d_2} ) \cdot  (0, \alpha^{*(q-1-d)}) =  \alpha^{*d_2} \cdot \alpha^{*(q-1-d)} = 0 \, . 
\]
\item Suppose $d = d_2$, then $q-b+ed = q - (c_1 + c_2)$ by Equation \eqref{MonomHirz} of Proposition \ref{prop:RRHirz}. Thus $(1,\alpha^{*(q-b+ed)}) \in \PRS_q(q-(c_1+c_2) + 1) = \PRS_q(c_1 + c_2)^\perp$ and
\[
 ( 0^{c_1}, \alpha^{*c_2} ) \cdot (1, \alpha^{*(q-b+ed)}) = 0 \, .
\]
\end{itemize}
In either case, $\ev_{\mathcal{H}_e}(M) \cdot v_d = 0$, whence $C_3 \subset C_e(a,b)^\perp$.

We now have the inclusion $C_1 + C_2 + C_3 \subset C_e(a,b)^\perp$. By puncturing $C_1 + C_2$ and $C_3$ to only keep the first row, we see that the punctured codes of $C_1 + C_2$ and $C_3$ are $\PRS_q(q-a)$ and $\sum_{d = \widetilde{s}+1}^a \langle (0, \alpha^{*(q-1-d)}) \rangle$ respectively, both of which are in direct sum. This implies that $C_1+C_2$ and $C_3$ are in direct sum as well, and thus $\dim(C_1 + C_2 + C_3) = \dim(C_1 + C_2) + \dim(C_3)$. In conclusion, we have
\begin{align*}
C_1 + C_2 + C_3 &\subset C_e(a,b)^\perp \, ,\\
\dim(C_1 + C_2 + C_3) &= \dim (C_{\mathbb{A},e}(a,b)^\perp) + 2q - \widetilde{s} \, .
\end{align*}
We end the proof by equality of dimensions using Corollary \ref{diffdimdual}.
\end{proof}

We conclude the study of $C_e(a,b)^\perp$ by expressing it as the sum of two AG codes on $\mathcal{H}_e$.

\begin{corollary}[Dual of $C_e(a,b)$ as a sum of two AG codes]\label{coro:dual}
Let $a$ and $b$ be two integers such that $0 \leq a \leq q-2$ and $ea \leq b$. Define two integers $a_\infty$ and $b_\infty$, large enough so that
\begin{align*}
ea_\infty &> (e+1)q-e-1-b \\
b_\infty &> q + e(q-a-1) \, .
\end{align*}
Then
\[
C_e(a,b)^\perp = C_e(q-a-1,b_\infty) + C_e(a_\infty, (e+1)q-e -1 -b) \, .
\]
\end{corollary}

\begin{proof}
 We will denote $C := C_e(a_\infty, (e+1)q - e -1 -b)$. Recall that since $b_\infty -e(q-a-1) > q$, the code $C_e(q-a-1,b_\infty)$ is the tensor product 
\[
 C_e(q-a-1,b_\infty ) =  \PRS_q(q-a) \otimes \F_q^{q+1} \, .
\]
With the same notations as Theorem \ref{theo:dual}, we consider the codes
\begin{align*}
C_1 &\coloneqq \overline{C_{\mathbb{A},e}}((q-1)\sigma_e + (q-2-b)F_e) \, , \\
C_2 &\coloneqq \PRS_q(q-a) \otimes \F_q^{q+1} \, , \\
C_3 &\coloneqq \langle (v_d)_{d \in \llbracket 1 + \widetilde{s},a \rrbracket} \rangle \, ;
\end{align*}
and we want to show that $C+C_2 = (C_1 + C_3) + C_2$. Looking at the basis of the Riemann--Roch space associated to $a_\infty S_e + ((e+1)q - e -1 -b)F_e$ (see Proposition \ref{prop:RRHirz}), by our choice of $x_\infty$ and since $x^q-x = 0$ for all $x \in \F_q$, we can show using Proposition \ref{prop: Code Hirzebruch} that 
\[
C = \bigoplus_{d = 0}^{q-1} \left[ \langle ( 0, \alpha^{*(q-1-d)}) \rangle \otimes  \PRS_q(q-b +ed) \right] \, ,
\]
where $\alpha^{*k}= (\alpha_1^k, \dots, \alpha_q^k)$ with $\F_q = \{ \alpha_1, \dots, \alpha_q\}$. Now, using Proposition \ref{prop:codeq} and after adding zero coordinates on the first row and first column, we write $C_1$ as the direct sum

 \begin{align*}
C_1 &= \overline{C_{\mathbb{A},e}}(q-1,(e+1)(q-1) - 1 - b)\\
& = \bigoplus_{d = 0}^{q-1} \left[ \langle (0, \alpha^{*(q-1-d)}) \rangle \otimes (0, \RS_q(q-1 -b + ed)) \right] \, .
\end{align*}
 Recall that for all $d \in \llbracket 1 + \widetilde{s},a \rrbracket$, we defined the vector $v_d \coloneqq (0, \alpha^{*(q-1-d)}) \otimes (1, \alpha^{*(q-b+ed)})$. Furthermore, for all integers $k$, we have $\PRS_q(k+1) = (1, \alpha^{*k}) + (0, \RS_q(k))$. It entails

\begin{align*}
C_1 + C_3 &= \bigoplus_{d = 0}^{\widetilde{s}} \left[ \langle (0, \alpha^{*(q-1-d)}) \rangle \otimes (0, \RS_q(q-1-b + ed)) \right] \\
		&\oplus \bigoplus_{d = \widetilde{s}+1}^{a} \left[ \langle (0, \alpha^{*(q-1-d)}) \rangle \otimes \PRS_q(q-b + ed) \right] \\
		& \oplus \bigoplus_{d = a+1}^{q-1} \left[ \langle (0, \alpha^{*(q-1-d)}) \rangle \otimes (0, \RS_q(q-1-b + ed)) \right] \, .
\end{align*}

However, if $d \leq \widetilde{s}$ then $q-b +ed \leq 0$ and $(0,\RS_q(q-1-b+ed)) = \PRS_q(q-b+ed) = \{ 0 \}$,  so we can simplify the sum as follow
\begin{align*}
C_1 + C_3 &= \bigoplus_{d = 0}^{a} \left[ \langle (0, \alpha^{*(q-1-d)}) \rangle \otimes \PRS_q(q-b + ed) \right] \\
		& \oplus \bigoplus_{d = a+1}^{q-1} \left[ \langle (0, \alpha^{*(q-1-d)}) \rangle \otimes (0, \RS_q(q-1-b + ed)) \right] \, .
\end{align*}
For all $d \in \llbracket a+1, q-1 \rrbracket$, define the vector $v_d \coloneqq (0, \alpha^{*(q-1-d)}) \otimes (1, \alpha^{*(q-b+ed)}) \in  C_2$. Adding these vectors to $C_1 + C_3$, we get
\begin{align*}
(C_1 + C_3) + \langle (v_d)_{d \in \llbracket a+1, q-1 \rrbracket} \rangle &= \bigoplus_{d = 0}^{q-1} \left[ \langle (0, \alpha^{*(q-1-d)}) \rangle \otimes \PRS_q(q-b + ed) \right] \\
&= C \, ,
\end{align*}
and since $ \langle (v_d)_{d \in \llbracket a+1, q-1 \rrbracket} \rangle$ is a subspace of $C_2$, we conclude

\[
C + C_2 = (C_1 + C_3) + C_2 = C_e(a,b)^\perp \, .
\]
\end{proof}

\section{Equivalence of codes and orthogonal codes on Hirzebruch surfaces}\label{section:equivalence}

In this section, we address the difficulty to obtain inclusions between codes of the form $C_e(a,b)$, due to $S_e$ and $F_e$ having $\F_q$-rational points in their respective supports (see Remark \ref{remarque chiante}). To this end, we express any AG code on $\mathcal{H}_e$ parametrized by a divisor $G$ containing no $\F_q$-rational point in its support. By choosing suitable parametrizing divisors, we will be able to define a new set of codes on $\mathcal{H}_e$ that allows for inclusions between codes. Recall by Proposition \ref{prop:RRgeneral} that we explicitly gave the Riemann--Roch space associated to \emph{any} divisor $G \in \Div(\mathcal{H}_e)$ and function $f \in \F_q(\mathcal{H}_e)$ such that $G + (f) = aS_e + bF_e$:
\[
L( G) = \left( \frac{f}{X_1^aT_1^{b}} \right)  L_*(aS_e + bF_e) \, .
\]
Everything that we present in this section is easily adaptable for AG codes defined on \emph{any} smooth and integral projective variety, regardless of the dimension, and is already well-known in the community. We simply phrase these notions in the context of Hirzebruch surfaces.
\subsection{Equivalence of AG codes}
We introduce here the notion of equivalence of codes, and the link between equivalence of AG codes and equivalence of their respective parametrizing divisors.

\begin{definition}[Schur Product]
Let $n \in \mathbb{N}$ and take two vectors $m = (m_1, \dots, m_n) \in \F_q^n$ and $m'=(m'_1, \dots, m'_n) \in \F_q^n$. The \emph{Schur product} of $m$ and $m'$, denoted $m*m'$, is their coordinate-wise multiplication
\[
m*m' = (m_1m'_1, \dots, m_nm'_n) \in \F_q^n \, .
\]
We denote $m^{*k} \coloneqq (m_1^k, \dots, m_n^k)$ for all integer $k$.
Let $C$ be a code over $\F_q$ of length $n$ and take $m \in \F_q^n$. We define
\[
m*C \coloneqq \{ m*c \, | \, c \in C \} \, .
\]
Finally, let $C$ be an AG code of length $n$ on an algebraic variety $\mathcal{V}$ constructed via the evaluation map $\ev$ on the points $P_1,\dots,P_n$. 
Then, for all function $f \in \F_q(\mathcal{V})$ regular in a neighborhood of $P_i$ for all $i \in \llbracket 1,n \rrbracket$, we define the code
\[
f*C = \ev(f)*C \,
\]
where $\ev(f) = (f(P_1), \dots, f(P_n))$.
\end{definition}

\begin{definition}[Equivalence of codes]
We say that two codes $C$ and $C'$ on $\F_q$ of length $n$ are \emph{equivalent} if there exists $x = (x_1, \dots, x_n) \in \left( \F_q^* \right)^n$ such that $x*C = C'$, \ie
\[ \forall (c_1,\dots,c_n) \in \F_q^n, \, (x_1c_1, \dots, x_nc_n) \in C' \Longleftrightarrow (c_1, \dots, c_n) \in C \, .\]
We will write $C \simeq C'$.
\end{definition}

Two equivalent codes trivially share the same parameters, see \cite[2.2.13]{Stic09}. The duals of two equivalent codes are also equivalent, entailing for instance that the dual distance is also the same.

\begin{proposition}
Let $C$ be a code over $\F_q$ of length $n$, and let $m \in \left( \F_q^* \right)^n$. Then
\[
 m^{*(q-2)} * C^\perp = \left( m * C \right)^\perp \,
\]
\end{proposition}

\begin{proof}
Since $m^{*(q-1)} = (1, \dots, 1)$, for all $x \in \F_q^n$ and $c \in C$ we have $(m^{*(q-2)} * x) \cdot (m*c) = x \cdot c$. 
This implies that $m^{*(q-2)}*C^\perp \subset (m*C)^\perp$, and the equality comes from equality of dimensions.
\end{proof}

\begin{proposition}[AG codes with equivalent divisors]\label{prop:AGequivalent}
 Let $\mathcal{S}$ be a smooth projective and geometrically integral surface over $\mathbb{F}_q$, and let $G \simeq G'$ be two linearly equivalent divisors on $\mathcal{S}$. Let $\Delta = \{ P_1, \dots, P_n \}$ be a set of $n$ distinct $\F_q$-rational points of $\mathcal{S}$ that avoid the supports of $G$ and $G'$, and define the two AG codes $C_G = \ev_\Delta (L(G))$ and $C_{G'} = \ev_\Delta(L(G'))$ of length $n$. Then $C_G$ and $C_{G'}$ are equivalent.
\end{proposition} 

\begin{proof}
Write $G = (f) + G'$ with $f \in \F_q(\mathcal{S})$. For all $i \in \llbracket 1, n \rrbracket$, since $P_i$ is not in the supports of $G$ and $G'$, the evaluation $f(P_i)$ is well-defined and non-zero, that is, $f(P_i) \in \F_q^*$. Thus, $C_{G'} = f* C_G$ with $\ev_\Delta(f) \in \left( \F_q^* \right)^n$.
\end{proof}

\subsection{Equivalent codes on Hirzebruch surfaces}

We now translate Proposition \ref{prop:AGequivalent} in the context of Hirzebruch surfaces.
\begin{proposition}\label{prop:Hirzequivalent}
Let $(a,b) \in \mathbb{N}^2$ and let $G \simeq aS_e + bF_e$ be a divisor on $\mathcal{H}_e$ that avoids every points of $\mathcal{H}_e(\F_q)$. Denote $(f_G) \coloneqq (aS_e + bF_e) - G$ with $f_G \in \F_q(\mathcal{H}_e)$, and consider the AG code $C_e(G)$ over $\F_q$ and of length $(q+1)^2$ defined by $C_e(G) \coloneqq \ev_{\mathcal{H}_e}( L(G))$. Then $C_e(a,b) \simeq C_e(G)$ and
\[
C_e( G) = \left( \frac{f_G}{X_1^aT_1^{b}} \right) * C_e(a, b) \, .
\]
\end{proposition}

\begin{proof}
Recall that $F_e = \{T_1 = 0 \}$ and $S_e = \{X_1 = 0 \}$. 
By construction of $f_G$ and the hypothesis on $G$, the evaluation of $ \left( \frac{f_G}{X_1^aT_1^{b}} \right)$ at every $\F_q$-rational point of $\mathcal{H}_e$ is well-defined and non-zero, whence
\[
\ev_{\mathcal{H}_e}  \left( \frac{f_G}{X_1^aT_1^{b}} \right) \in \left( \F_q^* \right)^{(q+1)^2} \, .
\]
Proposition \ref{prop:RRgeneral} allows us to conclude.
\end{proof}

Our first mean to obtain an inclusion between two AG codes on $\mathcal{H}_e$ is the following Corollary to Proposition \ref{prop:Hirzequivalent}:

\begin{corollary}\label{coro:Hirzequivalent} Let $G_1 \simeq a_1S_e + b_1F_e$ and $G_2 \simeq a_2S_e + b_2F_e$ be two divisors whose supports avoid $\mathcal{H}_e (\F_q)$ and such that $G_1 \leq G_2$. Then there exists two linear codes $C_1 \simeq C_e(a_1,b_1)$ and $C_2 \simeq C_e(a_2,b_2)$ such that
\[
C_1 \subset C_2 \, .
\]
\end{corollary}  

\begin{proof}
Since $G_1 \leq G_2$ then $L(G_1) \subset L(G_2)$ and we apply Proposition \ref{prop:Hirzequivalent} to $C_1 \coloneq C_e(G_1)$ and $C_2 \coloneq C_e(G_2)$.
\end{proof}

There always exists a divisor $G \simeq aS_e+bF_e$ whose support avoids any finite number of points, thanks to the well-known following Lemma.

\begin{lemma}[Moving lemma {\cite[III.1.3]{Shaf13}}]\label{movinglemma}
 Let $\mathcal{S}$ be a smooth projective and geometrically integral surface over $\mathbb{F}_q$. For all divisor $D \in \mathcal{S}$ and all finite families of points $P_1, \dots, P_n \in \mathcal{S}$, there exists $D' \in \Div(\mathcal{S})$ such that $D' \simeq D$ and $P_i \notin \Supp(D')$ for all $ i \in \llbracket 1,n \rrbracket$.
\end{lemma}

However, the assumptions of Corollary \ref{coro:Hirzequivalent} also includes an order relation $G_1 \leq G_2$ between our divisors, meaning $G_2 - G_1 \geq 0$. We have to construct, under good assumptions on $(a,b) \in \mathbb{N}^2$, a curve (\ie a positive divisor) $G \simeq aS_e+bF_e$ that avoids $\mathcal{H}_e(\F_q)$. This is the aim of the following Proposition.

\begin{proposition}\label{prop:equivdiv}
Let $(a,b) \in \mathbb{N}^2$. If $a \geq 2$ then there exists a positive divisor $\widetilde{\sigma}_{e,a} \simeq a\sigma_e$ whose support avoids the points of $\mathcal{H}_e(\F_q)$. Moreover, if $b-ea \geq 2$ there exists a positive divisors $\widetilde{F}_{e,b-ea} \simeq (b-ea)F_e$ whose support avoids $\mathcal{H}_e(\F_q)$.
\end{proposition}

\begin{proof}
Recall that $\sigma = \{X_2 = 0 \}$ and $F_e = \{T_1 = 0 \} \simeq \{T_2 = 0 \} = F'_e$. 

Suppose $a \geq 2$ and let $p$ be a prime number that divides $a$. Let $P(x) \in \F_q[x]$ be an irreducible polynomial on $\F_q$ of degree $p$, and consider the curve $\widetilde{\sigma}_{e,p} \coloneq \{ P(X_2) = 0 \}$. After a base field extension from $\F_q$ to $\F_{q^p}$, the curve $\widetilde{\sigma}_{e,p}$ splits into $p$ distinct curves, each equivalent to $\sigma_e$ and one for each distinct root of $P$. This implies that $\widetilde{\sigma}_{e,p}.F_e = p$ and since $\widetilde{\sigma}_{e,p} . \sigma_e = 0$ then $\widetilde{\sigma}_{e,p} \simeq p \sigma_e \,$ . By construction of $\widetilde{\sigma}_{e,p}$, no $\F_q$-rational point is in its support and we can conclude by defining 
\[ \widetilde{\sigma}_{e,a} \coloneq \frac{a}{p} \widetilde{\sigma}_{e,p} \simeq a \sigma_e \, . \]
We use the exact same proof for $F_e$ by considering the curve $\widetilde{F}_{e,p'} = \{ P'(T_2) = 0 \} \simeq p' F'_e$ for any prime number $p'$, where $P'$ is an irreducible polynomial on $\F_q$ of degree $p'$.
\end{proof}

\begin{remark}
When $e>0$, it is not possible to construct a curve $0 \leq \widetilde{S}_{e,b} \simeq b S_e$ for an integer $b > 0$ whose support avoids the $\F_q$-rational points of $\mathcal{H}_e$. Indeed, since $S_e$ has a negative self-intersecting number $-e$ then it is the only curve in its equivalent class and the same goes for $bS_e$.
\end{remark}

We can now constructs pairs of AG codes on $\mathcal{H}_e$, one containing the other, provided that the difference between their parameters is high enough (or zero).

\begin{theorem}\label{theo:equivcode}
Let $(\delta_a, \delta_b)$ be a pair of integers such that at least \emph{one} of the following holds:

\begin{itemize}
\item $\delta_a = 0$ and $\delta_b \geq 2$,
\item $\delta_a \geq 2$ and $\delta_b = 0$,
\item or $\delta_a \geq 2$ and $\delta_b - e\delta_a \geq 2$.
\end{itemize} 
Then, for all pair of integers $(a,b)$ there exists two codes $\widetilde{C}_e(a,b) \simeq C_e(a,b)$ and $\widetilde{C}_e(a+\delta_a , b+\delta_b) \simeq C_e(a+\delta_a , b+\delta_b)$ such that
\[
\widetilde{C}_e(a,b) \subset \widetilde{C}_e(a+\delta_a , b+\delta_b) \, .
\]
\end{theorem}

\begin{proof}
By the moving Lemma \ref{movinglemma}, there exists a divisor $G \simeq aS_e + bF_e$ on $\mathcal{H}_e$ with no $F_q$-rational point in its support. We then use Proposition \ref{prop:equivdiv} to construct a positive divisor $0 \leq \Delta \simeq \delta_a \sigma_e + (\delta_b-e \delta_a) F_e \simeq \delta_a S_e + \delta_b F_e$ that avoids $\mathcal{H}_e(\F_q)$. It holds
\[
aS_e + bF_e \simeq G \leq G + \Delta \simeq (a+\delta_a)S_e + (b+\delta_b)F_e \, .
\]
We conclude by Corollary \ref{coro:Hirzequivalent}:
\[
C_e(a,b) \simeq C_e(G) \subset C_e(G + \Delta) \simeq C_e(a+\delta_a, b + \delta_b) \, .
\]
\end{proof}

\begin{corollary}[Orthogonal codes on Hirzebruch surfaces]\label{coro:ortho}
Let $a,b \in \mathbb{N}$ be such that $a \leq q-2$ and $0 \leq b-ea$. For all pairs of positive integers $(a',b')$ such that $a \leq q-a'-3$ or $a = q-a'-1$,
there exists a code $C\simeq C_e(a,b)$ such that
\[
C \subset C_e(a',b')^\perp \, .
\]
\end{corollary}

\begin{proof}
From Theorem \ref{theo:dual}, we know that $\PRS_q(q-a') \otimes \F_q^{q+1} \subset C_e(a',b')^\perp$. However, by Proposition \ref{prop: Code Hirzebruch}, it holds $\PRS_q(q-a') \otimes \F_q^{q+1} = C_e(q-a'-1,B)$, where $B \in \mathbb{N}$ is large enough so that $B \geq b+2$ and $B  - e(q-a'-1) \geq q$. From Theorem \ref{theo:equivcode},  there exists a code $\widetilde{C}_e(a,b) \simeq C_e(a,b)$ such that
\[
\widetilde{C}_e(a,b) \subset C_e(q-a'-1,B) \subset  C_e(a',b')^\perp \, .
\]
\end{proof}

\section{Construction of CSS quantum codes from Hirzebruch surfaces}\label{section:css}
Quantum codes are tools to protect qudits from errors in the world of quantum computing. A $\llbracket n,k,d \rrbracket_q$ quantum code encodes $k$ logical qudits in $n$ physical qudits and can correct up to $\frac{d-1}{2}$ quantum errors.

In this section, we construct quantum codes using the CSS method, in the context of Hirzebruch surfaces. For the rest of this paper, we fix ${e \geq 2}$.

\subsection{The CSS construction}
The CSS construction, developed independently by Calderbank \& Shor \cite{calderbank1996good}, and Steane \cite{steane2002enlargement}, produces a quantum code from a pair of linear codes. We recall the construction below.

For any subset $E \subset \F_q^n  \setminus \{ 0 \}$, $n \in \mathbb{N}$, we denote $d(E) = \min \{\wt(x) \, | \, x \in E \}$.

\begin{theorem}[CSS construction \cite{calderbank1996good}]\label{th:css}
Let $C_1 \subset C_2 \subset \mathbb{F}_q^n$  be two linear codes. Then, there exists a $\llbracket n,k,d \rrbracket_q$ quantum code called CSS, and denoted by $(C_1,C_2)_{\CSS}$, where $k=\dim(C_2)-\dim(C_1)$ and $d=\min\{ d(C_2\setminus C_1),d(C_1^\perp\setminus C_2^\perp)\}$.
\end{theorem}

For any pair of codes $C_1 \subset C_2$, $d(C_2 \setminus C_1) \geq d(C_2)$, and the $\llbracket n,k,d\rrbracket_q$ code $(C_1,C_2)_{\CSS}$ is said to be \emph{degenerate} if $d > \min\{ d(C_2), d(C_1^\perp) \}$, and \emph{non-degenerate} otherwise. Note that $d(C_1) > d(C_2)$ implies $d(C_2 \setminus C_1) = d(C_2)$.

\subsection{CSS codes from $C_e(a,b)$}

We will now apply Theorem \ref{th:css} using linear codes equivalent to those of the form $C_e(a,b)$. We start by the injective case.
\begin{theorem}[CSS codes from Hirzebruch surfaces, injective case]\label{th:cssinj}
 Let $(a,b)$ be a pair of positive integers and $(\delta_a,\delta_b) \in (\mathbb{N} \setminus \{1 \})^2$ be such that $ a + \delta_a \leq q-1$, $ b + \delta_b \leq q-1$, $b - ea \geq 0$ and $\delta_b -e \delta_a \in \llbracket 0, q-1 \rrbracket \setminus \{ 1 \}$. Then there exists a $\llbracket n, k, d \rrbracket_q$ CSS quantum code, where
\begin{align*}
n&=(q+1)^2 \, ; \\
k &= (a + \delta_a +1)\left(b + \delta_b + 1 - \frac{e(a +\delta_a)}{2}\right) - (a +1)\left(b + 1 - \frac{ea}{2}\right) \, ; \\
d &\geq \min \{ a, b -ea \} +2  \, .
\end{align*}
\end{theorem}

\begin{proof}
Using Theorem \ref{theo:equivcode}, we construct $C_1 \simeq C_e(a,b)$ and $C_2 \simeq C_e(a + \delta_a, b + \delta_b)$ such that $C_1 \subset C_2$. Since, $b + \delta_b < q$, we have by Proposition \ref{prop:paramcode} and Theorem \ref{prop:dualdistance} that
\[
d(C_1^\perp) \leq \min \{a,b \} +2  < 2q \leq d(C_2) \, . 
\]

We now construct the code $(C_1,C_2)_{\CSS}$ using Theorem \ref{th:css}, whose parameters are computed using  Proposition \ref{prop:paramcode} and Theorem \ref{prop:dualdistance}.
\end{proof}

As expected, using the same notation as Theorem \ref{th:cssinj}, the dual distance is very low compared to the length and cannot go higher than $q+1$. This motivates us to leave the injective case to obtain bigger dimension for a fixed dual distance. Consider the code $C_e(m,(e+1)m)$ where $0 \leq m \leq q-2$. It is the smallest AG code on $\mathcal{H}_e(\F_q)$ whose dual distance is $m+2$. We fix $C_1 \simeq C_e(m,(e+1)m)$, and we are going to construct $C_2 \simeq C_e(a,b)$ such that $C_1 \subset C_2$ and $d(C_2) = d(C_1^\perp) = m+2$, while having the biggest possible dimension for $C_2$. Using Proposition \ref{prop:paramcode}, we see that to obtain $d(C_e(a,b)) = m+2$, we can fix $a = q-1$, and choose $b \in \llbracket (e+1)m, (e+1)(q-1) \rrbracket$ such that
\[
q-m-2 = \left \lfloor \frac{b-q}{e} \right \rfloor \, .
\]
To maximise $b$, we fix
\[
b = e(q-m-1) + q-1 \leq (e+1)(q-1) \, .
\]
It remains to show for which value of $m$ we have $b \geq (e+1)m$, which goes as follows:
\begin{align*}
b \geq (e+1)m &\Longleftrightarrow (e+1)q-e-1 \geq (2e+1)m \\
& \Longleftrightarrow m \leq \frac{(e+1)(q-1)}{2e+1} \, .
\end{align*}
This bound allows for non-trivial values of $m$ when $q$ is large enough compared to $e$. We formalize this result in the following theorem.

\begin{theorem}\label{th:css2}
Let $m$ be an integer such that

\[
 0 \leq m \leq  \min \left\{ \left \lfloor \frac{(e(q-1)-q+3}{2e+1} \right \rfloor , q-3 \right\}  \, .
\]
 Then, there exists a $\llbracket (q+1)^2, k_2 - k_1, \geq m+2 \rrbracket_q$ quantum code, where
\[
k_2 = (q-m-1)(q+1) + (m+1)\left( q - \frac{em}{2} \right) \,
\]
and
\[
 k_1 = (\widetilde{s}_m +1)(q+1) + (m-\widetilde{s}_m)\left( (e+1)m + 1 - e \left( \frac{m + \widetilde{s}_m + 1}{2} \right) \right) \, ,
\]
where $\widetilde{s}_m \coloneqq \max \{ -1, \lfloor \frac{(e+1)m-q}{e} \rfloor \}$.
\end{theorem}

\begin{proof}
As with the proof of Theorem \ref{th:cssinj}, consider the CSS code $(C_1,C_2)_{\CSS}$ with $C_1 \simeq C_e(m,(e+1)m)$ and $C_2 \simeq C_e(q-1,e(q-m-1)+q-1)$, such that $C_1 \subset C_2$ using Theorem \ref{theo:equivcode}. Again, the parameters of $(C_1,C_2)_{\CSS}$ come from Proposition \ref{prop:paramcode} and Theorem \ref{prop:dualdistance}.
\end{proof}

\begin{remark}\label{rem:degen}
A strong but necessary condition for the quantum code given in Theorem \ref{th:css2} to be degenerate is $d(C_e(m,(e+1)m)) = d(C_e(q-1,e(q-m-1)+q-1)^\perp) = m+2$.
\end{remark}

Recall that we fixed $e \geq 2$. By restricting $C_e(m,(e+1)m)$ to the injective cases, Theorem \ref{th:css2} can be simplified as follow.

\begin{corollary}
Let $m$ be an integer such that

\[
 0 \leq m \leq  \min \left\{ \left \lfloor \frac{(e(q-1)-q+3}{2e+1} \right \rfloor ,   \left \lfloor \frac{(q-1)}{e+1} \right \rfloor \right\}  \, .
\]
 Then there exists a $\llbracket (q+1)^2, k, m+2 \rrbracket_q$ quantum code, where
\[
k = q(q+1) - (e+1)m(m+1) \, .
\]
\end{corollary}

\begin{proof}
This is simply Theorem \ref{th:css2}, where we restrict to the case $\widetilde{s}_m = -1$. Furthermore, by Proposition \ref{prop:paramcode}, the minimum distance of $C_e(m,(e+1)m)$ cannot be $m+2$, meaning that the CSS code is non-degenerate by Remark \ref{rem:degen}.
\end{proof}

\subsection{CSS codes using orthogonal codes}

Using a pair $C_1$ and $C_2$ of orthogonal codes given by Corollary \ref{coro:ortho} and defined on an Hirzebruch surface, we construct the quantum codes $(C_1,C_2^\perp)_{\CSS}$. 

\begin{theorem}\label{th:css3}
Let $a_1,b_1 \in \mathbb{N}$ such that $a_1 \leq q-2$ and $0 \leq b_1-ea_1 \leq q-1$, and let $a_2,b_2 \in \mathbb{N}$ such that $a_1 \leq q-a_2-3$ and $0 \leq b_2-ea_2 \leq q-1$. Fix the following values 
\begin{align*}
\widetilde{s}_i &\coloneqq \max \left\{ -1, \left \lfloor \frac{q-b_i}{e} \right \rfloor \right \}, \, i \in \{ 1, 2 \} \, .
\end{align*}
Then, there exists a $\llbracket (q+1)^2, (q+1)^2 - k_2 - k_1, d \rrbracket_q$ quantum code where 
\[
k_i = (\widetilde{s}_i + 1)(q+1) + (a_i - \widetilde{s}_i)\left( b_i + 1 - e \left( \frac{a_i + \widetilde{s}_i + 1}{2} \right) \right), \, i \in \{1,2 \} \, 
\]
and
\[
d \geq \min\{ a_1,a_2,b_1-ea_1,b_2-ea_2 \} + 2 \, .
\]
\end{theorem}

\begin{proof}
The statement follows using Theorem \ref{th:css} with a pair of orthogonal codes from Corollary \ref{coro:ortho}, whose parameters are given in Proposition \ref{prop:paramcode} and Theorem \ref{prop:dualdistance}.
\end{proof}

\subsection{Conclusion}

As described by Voloch and Zarzar in \cite{voloch2009algebraic}, AG codes on surfaces tend to have a small dual distance which makes them good for decoding or recovering local properties of these codes, using curves imbedded in the surface. Unfortunately for us, this makes for a poor choice to construct CSS codes with, since the dual distance is always bounded from above by $q+1$ which is negligible compared to the length $(q+1)^2$. This phenomenon is expected on most surfaces. 

However, it might be possible to raise the dual distance of $C_e(a,b)$ and its CSS codes with a better knowledge of small weighted codewords of $C_e(a,b)^\perp$. Indeed, by puncturing $C_e(a,b)$ at good choices of rational points of $\mathcal{H}_e$, we can remove these codewords and raise the relative dual distance of the punctured code, \ie its dual distance divided by the length. Furthermore, with a good description of all codewords of $C_e(a,b)^\perp$ of minimal weight, we could construct degenerate CSS codes by choosing $(a_1,b_1)$ and $(a_2,b_2)$ such that $C_e(a_1,b_1)$ and $C_e(a_2,b_2)$ have the same distribution of small-weighted codewords. Experimentations with \texttt{SageMath} give some positive results when $\min \{ a_1,b_1 - ea_1 \} = \min \{ a_2, b_2 - e a_2 \}$, as shown in the following example.

\begin{example}
Let us fix $q=5$ and $e=2$, and consider $C \coloneqq C_2(3,7)$ and $C' \coloneqq C_2(4,9)$. Then $d(C^\perp) = d(C'^\perp) = 3$ and
\[
d(C^\perp \setminus C'^\perp) = 5 \, .
\]
\end{example}

\section*{Acknowledgments}
The author wishes to thank her supervisor Elena Berardini for suggesting to work on the topic of this paper and her guidance through it, as well as her co-supervisor Gilles Zémor for insightful discussions on quantum codes. The author also thanks Qing Liu and Alain Couvreur for their help regarding algebraic geometry. Finally, working on Hirzebruch surfaces was largely inspired by a conversation with Jade Nardi who then answered many questions on her work, and the author is grateful to her. 

This work was supported by the grants ANR-21-CE39-0009-BARRACUDA and ANR-22-CPJ2-0047-01 from the French National Research Agency.

No generative AI was used in the writing of this paper.

\bibliographystyle{ieeetr}
  \bibliography{Hirzebruch}
\end{document}